\documentclass[a4paper,
pagesize,
oneside, 				 
numbers=noenddot,
DIV=12,
headinclude=false, 
bibliography=totoc, 
fontsize=11pt, 
headsepline=true, 
open=right,
abstracton,
cleardoublepage=empty, 
footinclude=false,
english
]{scrartcl}

\usepackage{setspace}
\usepackage{standalone}

\usepackage{authblk}

\usepackage{mathtools}

\usepackage{stmaryrd}

\usepackage{pdflscape}

\DeclareSymbolFont{extraup}{U}{zavm}{m}{n}
\DeclareMathSymbol{\varclub}{\mathalpha}{extraup}{84} 
\DeclareMathSymbol{\varspade}{\mathalpha}{extraup}{85}

\usepackage{url}

\usepackage{upgreek}

\usepackage{sidecap}

\usepackage{xcolor}
\definecolor{webgreen}{rgb}{0,.5,0}
\definecolor{webbrown}{rgb}{.6,0,0}
\definecolor{RoyalBlue}{cmyk}{1, 0.50, 0, 0}
\definecolor{darkred}{rgb}{0.9,0.1,0.1}
\definecolor{darkblue}{rgb}{0,0,0.7}
\definecolor{darkgreen}{rgb}{0,0.5,0}
\definecolor{bluegray}{rgb}{0.4, 0.6, 0.8}
\definecolor{cadmiumorange}{rgb}{0.93, 0.53, 0.18}
\definecolor{darkcerulean}{rgb}{0.03, 0.27, 0.49}
\usepackage[colorlinks=true, linktocpage=true, linkcolor=darkblue, citecolor=webgreen, backref=page, hyperfootnotes=true, bookmarks=false]{hyperref}

\usepackage[english]{babel}
\usepackage[T1]{fontenc} 
\usepackage{lmodern} 
\usepackage{amsmath, amsfonts, amssymb}
\usepackage{mathtools}
\usepackage{mathrsfs}
\usepackage{verbatim}
\usepackage{stmaryrd}
\usepackage[framed,amsmath,thmmarks,thref,hyperref]{ntheorem}
\usepackage{commath}

\usepackage{longtable}
\usepackage{booktabs}
\usepackage{caption}

\usepackage[labelfont=bf]{caption}

\allowdisplaybreaks

\usepackage[english=american]{csquotes}

\usepackage{graphicx}

\usepackage{comment}

\usepackage{framed}

\usepackage{pdflscape}

\usepackage[]{listofsymbols}

\setkomafont{sectioning}{\bfseries} 

\setkomafont{pagenumber}{\bfseries}
\setkomafont{pageheadfoot}{\normalfont}
\setkomafont{pagenumber}{\normalfont}

\usepackage{enumitem} 
\usepackage{cleveref}
\usepackage{ifthen}

\usepackage{tikz,pgfplots}
\pgfplotsset{compat=newest}
\usetikzlibrary{cd}
\usetikzlibrary{calc}
\usetikzlibrary{arrows, scopes}
\usetikzlibrary{positioning}
\usetikzlibrary{shapes.geometric}
\usetikzlibrary{backgrounds}
\usetikzlibrary{arrows,chains,matrix,positioning,scopes}

\usepackage[multiple,hang,flushmargin]{footmisc}

\usepackage{graphicx}

\usepackage[colorinlistoftodos,disable]{todonotes} 

\usepackage{marginnote}

\usepackage[framemethod=TikZ]{mdframed}

\mdfdefinestyle{temporary}{linewidth=3pt, linecolor=red}
\mdfdefinestyle{onlyframe}{middlelinecolor=black!100,roundcorner=1ex, skipbelow=15pt}
\mdfdefinestyle{theorem}{middlelinecolor=black!25, backgroundcolor=black!15,roundcorner=1ex, skipbelow=15pt}
\mdfdefinestyle{onlyback}{middlelinecolor=white!100, backgroundcolor=black!10,roundcorner=1ex, skipbelow=15pt}
\mdfdefinestyle{main}{middlelinecolor=darkblue!25, backgroundcolor=darkblue!15,roundcorner=1ex, skipbelow=15pt}

\usepackage{commath}

\usepackage{enumitem}

\usepackage{nicefrac}

\usepackage{rotating}

\usepackage{relsize}

\usepackage{ctable}
\usepackage{multicol}
\usepackage{multirow}

\usepackage{accents}

\usepackage{ucs}
\usepackage{lmodern} 
\usepackage{setspace,graphicx,tikz,tabularx} 
\usepackage[draft=false,babel,tracking=true,kerning=true,spacing=true]{microtype} 

\input{private.sty}

\usepackage{tikz,pgfplots}
\pgfplotsset{width=\textwidth,height=5cm}
\usepackage{mhequ}
\usepackage{mhsymb}

\usepackage{mathrsfs}

\usepackage{multirow}

\usepackage[framemethod=TikZ]{mdframed}

\mdfdefinestyle{temporary}{linewidth=3pt, linecolor=red}
\mdfdefinestyle{theorem}{middlelinecolor=black!25, backgroundcolor=black!15,roundcorner=1ex, skipbelow=15pt}
\mdfdefinestyle{main}{middlelinecolor=darkblue!25, backgroundcolor=darkblue!15,roundcorner=1ex, skipbelow=15pt}

\def\CA{\mathcal{A}}

\def\CC{\mathcal{C}}
\def\CD{\mathcal{D}}

\def\CF{\mathcal{F}}
\def\CG{\mathcal{G}}
\def\CH{\mathcal{H}}
\def\CI{\mathcal{I}}

\def\CL{\mathcal{L}}
\def\CM{\mathcal{M}}
\def\CN{\mathcal{N}}

\def\CP{\mathcal{P}}

\def\CR{\mathcal{R}}
\def\CS{\mathcal{S}}
\def\CT{\mathcal{T}}
\def\CU{\mathcal{U}}

\def\CX{\mathcal{X}}

\def\DD{\CD}

\def\FF{\mathscr{F}}

\def\II{\mathscr{I}}

\def\LL{\mathscr{L}}
\def\MM{\boldsymbol{\CM}}

\def\TT{\mathscr{T}}

\def\fC{\mathfrak{C}}

\def\fc{\mathfrak{c}}

\def\ft{\mathfrak{t}}

\def\e{\operatorname{e}}
\def\t{\operatorname{t}}

\def\sh{\mathsf{h}}

\newcommand{\esh}{\operatorname{e}_{\mathsf{h}}}

\newcommand*\oline[1]{%
	\vbox{%
		\hrule height 0.5pt
		\kern0.25ex
		\hbox{%
			\kern-0.1em
			\ifmmode#1\else\ensuremath{#1}\fi
			\kern-0.1em
		}
	}
}

\colorlet{testcolor}{green!60!black}
\colorlet{testcolor2}{blue!100!black}
\colorlet{lightblue}{darkblue!80}
\colorlet{grayblue}{bluegray!100}
\colorlet{orange}{cadmiumorange!100}

\usetikzlibrary{calc}
\usetikzlibrary{shapes.misc}
\usetikzlibrary{shapes.symbols}
\usetikzlibrary{shapes.geometric}
\usetikzlibrary{snakes}
\usetikzlibrary{decorations}
\usetikzlibrary{decorations.markings}

\makeatletter
\pgfdeclareshape{crosscircle}
{
	\inheritsavedanchors[from=circle] 
	\inheritanchorborder[from=circle]
	\inheritanchor[from=circle]{north}
	\inheritanchor[from=circle]{north west}
	\inheritanchor[from=circle]{north east}
	\inheritanchor[from=circle]{center}
	\inheritanchor[from=circle]{west}
	\inheritanchor[from=circle]{east}
	\inheritanchor[from=circle]{mid}
	\inheritanchor[from=circle]{mid west}
	\inheritanchor[from=circle]{mid east}
	\inheritanchor[from=circle]{base}
	\inheritanchor[from=circle]{base west}
	\inheritanchor[from=circle]{base east}
	\inheritanchor[from=circle]{south}
	\inheritanchor[from=circle]{south west}
	\inheritanchor[from=circle]{south east}
	\inheritbackgroundpath[from=circle]
	\foregroundpath{
		\centerpoint%
		\pgf@xc=\pgf@x%
		\pgf@yc=\pgf@y%
		\pgfutil@tempdima=\radius%
		\pgfmathsetlength{\pgf@xb}{\pgfkeysvalueof{/pgf/outer xsep}}%
		\pgfmathsetlength{\pgf@yb}{\pgfkeysvalueof{/pgf/outer ysep}}%
		\ifdim\pgf@xb<\pgf@yb%
		\advance\pgfutil@tempdima by-\pgf@yb%
		\else%
		\advance\pgfutil@tempdima by-\pgf@xb%
		\fi%
		\pgfpathmoveto{\pgfpointadd{\pgfqpoint{\pgf@xc}{\pgf@yc}}{\pgfqpoint{-0.707107\pgfutil@tempdima}{-0.707107\pgfutil@tempdima}}}
		\pgfpathlineto{\pgfpointadd{\pgfqpoint{\pgf@xc}{\pgf@yc}}{\pgfqpoint{0.707107\pgfutil@tempdima}{0.707107\pgfutil@tempdima}}}
		\pgfpathmoveto{\pgfpointadd{\pgfqpoint{\pgf@xc}{\pgf@yc}}{\pgfqpoint{-0.707107\pgfutil@tempdima}{0.707107\pgfutil@tempdima}}}
		\pgfpathlineto{\pgfpointadd{\pgfqpoint{\pgf@xc}{\pgf@yc}}{\pgfqpoint{0.707107\pgfutil@tempdima}{-0.707107\pgfutil@tempdima}}}
	}
}
\makeatother

\colorlet{symbols}{black!90!black}
\colorlet{cameron}{darkgreen!90!black}
\colorlet{cameron2}{violet!90!black}
\colorlet{cameron3}{blue!90!black}
\colorlet{testcolor}{green!60!black}
\colorlet{darkblue}{blue!60!black}
\colorlet{darkgreen}{green!60!black}

\def\symbol#1{\textcolor{symbols}{#1}}
\def\1{\mathbf{\symbol{1}}}
\def\X{\symbol{X}}

\usetikzlibrary{calc}
\usetikzlibrary{shapes.misc}
\usetikzlibrary{shapes.symbols}
\usetikzlibrary{shapes.geometric}
\usetikzlibrary{snakes}
\usetikzlibrary{decorations}
\usetikzlibrary{decorations.markings}

\def\drawx{\draw[-,solid] (-3pt,-3pt) -- (3pt,3pt);\draw[-,solid] (-3pt,3pt) -- (3pt,-3pt);}
\tikzset{
	root/.style={draw=symbols,circle,inner sep=0pt,minimum size=0.5mm,fill=white},	
	smalldot/.style={circle,fill=symbols,draw=symbols,inner sep=0pt,minimum size=0.5mm},
	dot/.style={circle,fill=black,inner sep=0pt,minimum size=1mm},
	bigdot/.style={circle,fill=black,inner sep=0pt,minimum size=4mm},
	noiseedge/.style={black,semithick,decorate, decoration={snake,segment length=4pt,amplitude=1pt}},
	smallestdot1/.style={circle,fill=symbols,draw=symbols,inner sep=0pt,minimum size=0.2mm},
	smallestdot2/.style={circle,fill=cameron,draw=cameron,inner sep=0pt,minimum size=0.2mm},
	smallnoiseedge1/.style={draw=symbols,decorate, decoration={snake,segment length=1.75pt,amplitude=0.55pt}},
	smallnoiseedge2/.style={draw=cameron,decorate, decoration={snake,segment length=1.75pt,amplitude=0.55pt}},
	noiseedge1/.style={draw=symbols,decorate, decoration={snake,segment length=0.9pt,amplitude=0.55pt}},
	noiseedge2/.style={draw=cameron,decorate, decoration={snake,segment length=0.9pt,amplitude=0.55pt}},
	noiseedge3/.style={draw=cameron2,decorate, decoration={snake,segment length=0.9pt,amplitude=0.55pt}},
	noiseedge4/.style={draw=cameron3,decorate, decoration={snake,segment length=0.9pt,amplitude=0.55pt}},
	smallnoisenode1/.style={draw=symbols,circle,inner sep=0pt,minimum size=0.5mm,fill=white},	
	smallnoisenode2/.style={draw=cameron,circle,inner sep=0pt,minimum size=0.5mm,fill=white},
	poly/.style={draw=symbols,circle,inner sep=0pt,minimum size=0.5mm,fill=black},
	dot/.style={circle,fill=black,inner sep=0pt,minimum size=1mm},
	int/.style={circle,fill=black,draw=black,inner sep=0pt,minimum size=0.7mm},
	circ/.style={circle,draw=black,inner sep=0pt, minimum size=1mm},
	var/.style={circle,fill=black!10,draw=black,inner sep=0pt, minimum size=2mm},
	dotred/.style={circle,fill=black!50,inner sep=0pt, minimum size=2mm},
	generic/.style={semithick,shorten >=1pt,shorten <=1pt},
	oddfunc/.style={generic, dotted},
	dist/.style={ultra thick,draw=testcolor,shorten >=1pt,shorten <=1pt},
	testfcn/.style={ultra thick,testcolor,shorten >=1pt,shorten <=1pt,<-},
	testfunction/.style={ultra thick,testcolor,shorten >=1pt,shorten <=1pt},
	testfcnx/.style={ultra thick,testcolor,shorten >=1pt,shorten <=1pt,<-,
		postaction={decorate,decoration={markings,mark=at position 0.6 with {\drawx}}}},
	kprime/.style={semithick,shorten >=1pt,shorten <=1pt,densely dashed,->},
	kprimex/.style={semithick,shorten >=1pt,shorten <=1pt,densely dashed,->,
		postaction={decorate,decoration={markings,mark=at position 0.4 with {\drawx}}}},
	kernel/.style={semithick,shorten >=1pt,shorten <=1pt,->,draw=black},
	Pkernel/.style={ultra thick,shorten >=1pt,shorten <=1pt,->,draw=blue},
	PkernelBig/.style={very thick,shorten >=1pt,shorten <=1pt,decorate, draw=blue, decoration={zigzag,amplitude=1.5pt,segment length = 3pt,pre length=2pt,post length=2pt}},
	multx/.style={shorten >=1pt,shorten <=1pt,
		postaction={decorate,decoration={markings,mark=at position 0.5 with {\drawx}}}},
	kernelx/.style={semithick,shorten >=1pt,shorten <=1pt,->,
		postaction={decorate,decoration={markings,mark=at position 0.4 with {\drawx}}}},
	kernel1/.style={->,semithick,shorten >=1pt,shorten <=1pt,postaction={decorate,decoration={markings,mark=at position 0.45 with {\draw[-] (0,-0.1) -- (0,0.1);}}}},
	kernel2/.style={->,semithick,shorten >=1pt,shorten <=1pt,postaction={decorate,decoration={markings,mark=at position 0.45 with {\draw[-] (0.05,-0.1) -- (0.05,0.1);\draw[-] (-0.05,-0.1) -- (-0.05,0.1);}}}},
	kernelBig/.style={semithick,shorten >=1pt,shorten <=1pt,decorate, decoration={zigzag,amplitude=1.5pt,segment length = 3pt,pre length=2pt,post length=2pt}},
	kernelBigg/.style={thick,shorten >=1pt,shorten <=1pt,decorate, decoration={zigzag,amplitude=3.5pt,segment length = 7pt,pre length=2pt,post length=2pt}},
	kernelBigg1/.style={thick,shorten >=1pt,shorten <=1pt,decorate, decoration={saw,amplitude=3.5pt,segment length = 7pt,pre length=2pt,post length=2pt}},
	kernelBigg2/.style={thick,shorten >=1pt,shorten <=1pt,decorate, decoration={bumps,amplitude=3.5pt,segment length = 7pt,pre length=2pt,post length=2pt}},
	rho/.style={dotted,semithick,shorten >=1pt,shorten <=1pt},
	renorm/.style={shape=circle,fill=white,inner sep=1pt},
	labl/.style={shape=rectangle,fill=white,inner sep=1pt},
	cumu2n/.style={inner sep=3pt},
	cumu2/.style={draw=red!80,fill=red!40},
	cumu3/.style={regular polygon, regular polygon sides=3,draw=red!80,rounded corners=3pt,fill=red!40,minimum size=5mm},
	cumu4/.style={regular polygon, regular polygon sides=4,draw=red!80,rounded corners=3pt,fill=red!40,minimum size=7mm},
	cumu5/.style={regular polygon, regular polygon sides=5,draw=red!80,rounded corners=3pt,fill=red!40,minimum size=7mm},
	bcumu2n/.style={inner sep=3pt},
	bcumu2/.style={draw=blue!80,fill=blue!40},
	bcumu3/.style={regular polygon, regular polygon sides=3,draw=blue!80,rounded corners=3pt,fill=blue!40,minimum size=5mm},
	bcumu4/.style={regular polygon, regular polygon sides=4,draw=blue!80,rounded corners=3pt,fill=blue!40,minimum size=7mm},
	bcumu5/.style={regular polygon, regular polygon sides=5,draw=blue!80,rounded corners=3pt,fill=blue!40,minimum size=7mm},
	xi/.style={circle,fill=symbols!30,draw=symbols,inner sep=0pt,minimum size=1.2mm},
	cmn/.style={circle,fill=cameron!50,draw=symbols,inner sep=0pt,minimum size=1.2mm},
	xix/.style={crosscircle,fill=symbols!10,draw=symbols,inner sep=0pt,minimum size=1.2mm},
	xib/.style={circle,fill=symbols!10,draw=symbols,inner sep=0pt,minimum size=1.6mm},
	xibx/.style={crosscircle,fill=symbols!10,draw=symbols,inner sep=0pt,minimum size=1.6mm},
	not/.style={circle,fill=symbols,draw=symbols,inner sep=0pt,minimum size=0.5mm},
	>=stealth,
}
\makeatletter
\def\DeclareSymbol#1#2#3{\expandafter\gdef\csname MH@symb@#1\endcsname{\tikz[baseline=#2,scale=0.13,draw=symbols]{#3}}\expandafter\gdef\csname MH@symb@#1s\endcsname{\scalebox{0.7}{\tikz[baseline=#2,scale=0.15,draw=symbols]{#3}}}}
\def\<#1>{\csname MH@symb@#1\endcsname}
\makeatother

\DeclareSymbol{wn}{0}{
	\draw[white] (-.5,0) -- (.5,0); \draw[noiseedge1] (0,0)  -- (0,1.5) {};
}

\DeclareSymbol{cm}{0}{
	\draw[white] (-.5,0) -- (.5,0); \draw[noiseedge2] (0,0)  -- (0,1.5) {};
}

\DeclareSymbol{cm2}{0}{
	\draw[white] (-.5,0) -- (.5,0); \draw[noiseedge3] (0,0)  -- (0,1.5) {};
}

\DeclareSymbol{cm3}{0}{
	\draw[white] (-.5,0) -- (.5,0); \draw[noiseedge4] (0,0)  -- (0,1.5) {};
}

\DeclareSymbol{1}{0}{
	\draw (0,0) -- (.7,1.3) node[smallestdot1] {};
	\draw[noiseedge1] (.7,1.3) -- (-.4,2.1) {};
	\node[root] (root) at (0,0) {};	
}

\DeclareSymbol{1g}{0}{
	\draw[noiseedge2] (.7,1.3) -- (-.4,2.1) {};
	\draw (0,0) -- (.7,1.3) node[smallestdot1] {};	
	\node[root] (root) at (0,0) {};	
}

\DeclareSymbol{1p}{0}{
	\draw[noiseedge3] (.7,1.3) -- (-.4,2.1) {};
	\draw (0,0) -- (.7,1.3) node[smallestdot1] {};	
	\node[root] (root) at (0,0) {};	
}

\DeclareSymbol{11}{0}{
	\draw[noiseedge1] (0,0) -- (-1.1,0.8)  {};
	\draw[noiseedge1] (.7,1.3) -- (-.4,2.1) {};
	\draw (0,0) -- (.7,1.3) node[smallestdot1] {};
	\node[root] (root) at (0,0) {};	
}

\DeclareSymbol{1g1}{0}{
	\draw[noiseedge1] (0,0) -- (-1.1,0.8)  {};
	\draw[noiseedge2] (.7,1.3) -- (-.4,2.1)  {};
	\draw (0,0) -- (.7,1.3) node[smallestdot1] {};
	\node[root] (root) at (0,0) {};	
}

\DeclareSymbol{1g1g}{0}{
	\draw[noiseedge2] (0,0) -- (-1.1,0.8) {};
	\draw[noiseedge2] (.7,1.3) -- (-.4,2.1) {};
	\draw (0,0) -- (.7,1.3) node[smallestdot1] {};
	\node[root] (root) at (0,0) {};	
}

\DeclareSymbol{1p1g}{0}{
	\draw[noiseedge2] (0,0) -- (-1.1,0.8) {};
	\draw[noiseedge3] (.7,1.3) -- (-.4,2.1) {};
	\draw (0,0) -- (.7,1.3) node[smallestdot1] {};
	\node[root] (root) at (0,0) {};	
}

\DeclareSymbol{1g1p}{0}{
	\draw[noiseedge3] (0,0) -- (-1.1,0.8) {};
	\draw[noiseedge2] (.7,1.3) -- (-.4,2.1) {};
	\draw (0,0) -- (.7,1.3) node[smallestdot1] {};
	\node[root] (root) at (0,0) {};	
}

\DeclareSymbol{11g}{0}{
	\draw[noiseedge2] (0,0) -- (-1.1,0.8) {};
	\draw[noiseedge1] (.7,1.3) -- (-.4,2.1)  {};
	\draw (0,0) -- (.7,1.3) node[smallestdot1] {};
	\node[root] (root) at (0,0) {};	
}

\DeclareSymbol{phi4}{0}{
	\draw (0,0) -- (1,1) node[smallestdot1] {};
	\draw[noiseedge2] (0,1.2) -- (0.7,1.9) {};
	\draw (0,0) -- (0,1.2) node[smallestdot1] {};
	\draw[noiseedge1] (1.2,0) -- (1.9,0.7) {};
	\draw (0,0) -- (1.2,0) node[smallestdot1] {};
	\draw[noiseedge1] (2,2) -- (2.7,2.7) {};
	\draw (1,1) -- (2,2) node[smallestdot1] {};
	\draw[noiseedge1] (1,2.2) -- (1.7,2.9) {};
	\draw (1,1) -- (1,2.2) node[smallestdot1] {};
	\draw[noiseedge2] (2.2,1) -- (2.9,1.7) {};	
	\draw (1,1) -- (2.2,1) node[smallestdot1] {};
	\node[root] (root) at (0,0) {};	
}

\DeclareSymbol{kpz}{0}{
	\draw[noiseedge1] (0,1.9) -- (0.7,2.6) {};
	\draw[gray] (0,1.2) -- (0,1.9) node[smallestdot1] {};
	\draw[noiseedge2] (0.7,1.2) -- (1.4,1.9) {};
	\draw[gray] (0,1.2) -- (0.7,1.2) node[smallestdot1] {};
	\draw[gray] (0,0) -- (0,1.2) node[smallestdot1] {};
	\draw[noiseedge2] (1.2,0.7) -- (1.9,1.4) {};
	\draw[gray] (1.2,0) -- (1.2,0.7) node[smallestdot1] {};
	\draw[noiseedge1] (1.9,0) -- (2.6,0.7) {};
	\draw[gray] (1.2,0) -- (1.9,0) node[smallestdot1] {};
	\draw[gray] (0,0) -- (1.2,0) node[smallestdot1] {};
	\node[root] (root) at (0,0) {};	
}

\usepackage{caption}
\usepackage{subcaption}

\usepackage{lipsum}

\usepackage{cite}

\setkomafont{dictumauthor}{\normalfont}
\setkomafont{dictumtext}{\normalfont}

\usepackage{colortbl}

\allowdisplaybreaks[1]

\usepackage{scalerel}

\makeatletter
\newcommand\bigwp{\mathop{\mathpalette\bigDi@mond\relax}}
\newcommand\bigDi@mond[2]{%
	\vcenter{\hbox{\m@th
			\scalebox{\ifx#1\displaystyle 2\else1.2\fi}{$#1\diamond$}%
	}}%
}
\makeatother

\author{Tom Klose}

\date{3 December 2024}

\title{The constant coefficient \\ in precise Laplace asymptotics for gPAM}

\numberwithin{equation}{section}

\begin{document}

\maketitle

\thispagestyle{plain}

\newcolumntype{R}[1]{>{\raggedright\arraybackslash}p{#1}}

\vspace{2em}

\noindent

\begin{abstract}
	This article resumes the analysis of precise Laplace asymptotics for the generalised Parabolic Anderson Model (gPAM) initiated by Peter Friz and the author. More precisely, we provide an explicit formula for the constant coefficient in the asymptotic expansion in terms of traces and Carleman--Fredholm determinants of certain explicit operators under only slightly stronger assumptions. The proof combines classical Gaussian analysis in abstract Wiener spaces with arguments from the theory of regularity structures. As an ingredient, we prove that the minimiser in the (extended) phase functional of gPAM has better than just Cameron--Martin regularity.
\end{abstract}

\setcounter{tocdepth}{4}
\tableofcontents

\section{Introduction}

In~\cite{friz_klose_22}, Peter Friz and the author have studied precise Laplace asymptotics of the generalised Parabolic Anderson Model~(gPAM). We recall the main result:

\begin{theorem} \label{thm:precise_laplace}
	Let~$\eps \in [0,1]$, $h \in \CH := L^2(\T^2)$, and assume that~$g \in C_b^{N+7}(\R,\R)$ for some $N \in \N_0$. 
	\label{e:solution_space}
	For some~\mbox{$\eta \in (\nicefrac{1}{2},1)$}, we denote by~$\hat{u}^\eps_h \in \CX_T := C_T \CC^\eta(\T^2)$ the \emph{renormalised} solution\footnote{See Hairer~\cite[Sec.~$9.3$ and~$10.4$]{hairer_rs} as well as Cannizzaro, Friz, and Gassiat~\cite[Sec.~$3.5$]{cfg} for a correct interpretation of the renormalisation procedure.} to gPAM driven by~$\eps \xi + h$ for a~$2$D space white noise~$\xi$, that is
	\begin{equation} \label{eq:gpam_shifted_eps}
		(\partial_t - \Delta) \hat{u}^\eps_h = g(\hat{u}^\eps_h) \sbr[1]{\eps \xi + h -\eps^2 g'(\hat{u}^\eps_h) \infty}
	\end{equation}
	with initial condition~$\hat{u}_h^\eps(0,\cdot) = u_0 \in \CC^\eta(\T^2)$. Let~$T^\eps_h$ be the \emph{explosion time}\footnote{See~\cite[App.~B]{friz_klose_22} for details. Sufficient conditions on~$g$ for non-explosion have been investigated in~\cite[Sec.~$3.8$]{cfg}, so the reader can think of~$T^\eps \equiv \infty$ for now.} of~$\hat{u}_h^\eps$, set~$\hat{u}^\eps := \hat{u}^\eps_0$, and ac\-cor\-ding\-ly~\mbox{$T^\eps := T_0^\eps$}. 
	Furthermore, assume that~$F: C_T \CC^\eta(\T^2) \to \R$ satisfies the following hypotheses: 
	\begin{enumerate}[label={(H\arabic*)},itemsep=5pt]
		\item \label{ass:h1} $F \in C_{b}\del[0]{C_T \CC^\eta(\T^2);\R}$, i.e. $F$ is continuous and bounded from $C_T \CC^\eta(\T^2)$ into $\R$.  
		\item \label{ass:h2} The functional $\FF$ given by 
		\begin{equation}
			\FF := (F \circ \Phi \circ \LL) + \II: \ \CH \to (-\infty,+\infty]
		\end{equation}
		attains its \emph{unique} minimum at $\sh \in \CH$. Here, $\II$ denotes \emph{Schilder's rate function}, i.e. 
		\begin{equation}
			\II(h) = \frac{\norm{h}^2_\CH}{2} \thinspace \mathbf{1}_{h \in \CH}  + \infty \thinspace \mathbf{1}_{h \notin \CH},
			\label{eq:schilders_rf}
		\end{equation}
		$\LL$ is the~\emph{lifting operator} from~$\CH$ into model space~$\MM$, and~$\Phi: \MM \to C_T \CC^\eta(\T^2)$ denotes the~\emph{abstract solution map} associated to~\eqref{eq:gpam_shifted_eps}.  	
		\item \label{ass:h3} The functional $F$ is $N + 3$ times Fr\'{e}chet differentiable in a neighbourhood $\CN$ of $(\Phi \circ \LL)(\sh)$. 
		In addition, there exist constants $M_1,\ldots,M_{N+3}$ such that
		\begin{equation}
			\abs[0]{D^{(k)}F(v)[y,\ldots,y]} \leq M_k \norm{y}_{\CX_T}^k, \quad k=1,\ldots,N+3,
			\label{thm:laplace_asymp:ass:h3:boundedness}
		\end{equation}
		holds for all~$v \in \CN$ and $y \in C_T \CC^\eta(\T^2)$.
		\item \label{ass:h4} The minimiser $\sh$ is \emph{non-denegerate}. More precisely, $D^2\FF\sVert_\sh > 0$ in the sense that for all $h \in \CH\setminus\{0\}$ we have
		\begin{equation}
			D^2\FF\sVert_{\sh}[h,h] > 0.
		\end{equation}
	\end{enumerate}	
	Then, the asymptotic expansion
	\begin{equation}
		J(\eps)
		\equiv
		\E\sbr{\exp\del{-\frac{F(\hat{u}^{\eps})}{\eps^2}} \mathbf{1}_{T^\eps > T}} = \exp\del[3]{-\frac{\FF(\sh)}{\eps^2}}\sbr{a_0 + \sum_{k=1}^N a_k \eps^k + {\scriptstyle\mathcal O}(\eps^N)} \quad \text{as} \; \, \eps  \to 0
		\label{thm:laplace_asymp:expansion}
	\end{equation}
	is true, where 
	\begin{equation}
		a_0 = \E\sbr[2]{\exp\del[2]{-\frac{1}{2} \hat{Q}_\sh}}, 
		\quad \hat{Q}_\sh := \partial_\eps^2\sVert[0]_{\eps = 0} F(\hat{u}^\eps_\sh).
		\label{thm:eq:coeff:a_0}
	\end{equation} 	
\end{theorem} 

\paragraph*{Main results.}

The article at hand is concerned with the analysis of~$a_0$ and aims to establish a formula that is as explicit as possible and therefore facilitates computations~\emph{in practice}. The following is our main result: 
\begin{theorem} \label{thm:main_result}
	Let~$\kappa > 0$ be small enough. 
	Let~$N = 0$ and take~$F,g$ as in~Theorem~\ref{thm:precise_laplace} with~$\eta = 1 - 5\kappa$ and for~$T > 0$ small enough.\footnote{The precise condition is that~$T < \bar{T}_\star$ for~$\bar{T}_\star > 0$ as in~Corollary~\ref{lemma:Nzero}.} 
	Additionally, assume that
	$u_0 \in \CC^{1+2\kappa}(\T^2)$.\footnote{See~Remark~\ref{rmk:reg_u_0} for an explanation of this regularity assumption and its relation to~Theorem~\ref{thm:reg_minimiser}.}
	Then, the coefficient~$a_0$ has the representation
	\begin{equation}
		a_0 = 
		e^{-\frac{1}{2}\sbr[0]{\operatorname{Tr}(A-\tilde{A}) + \lambda}} \operatorname{det}_2(\operatorname{Id}_\CH + A)^{-\nicefrac{1}{2}}
		\label{thm:eq_a0}
	\end{equation}
	where~$\operatorname{det}_2$ denotes the Carleman--Fredholm determinant\footnote{See~Dunford and Schwartz~\cite[Def.~$21$, p.~$1106$]{dunford_schwartz} or Simon~\cite[Thm.~$9.2$]{simon_trace_ideals} for a definition and properties of the determinant. One can think of~\enquote{$\operatorname{det}_2(\operatorname{Id}_\CH + A) = \operatorname{det}(\operatorname{Id}_\CH + A) e^{-\operatorname{Tr}(A)}$} where~$\operatorname{det}$ is the usual Fredholm determinant for trace-class operators: Morally, the singularities of the two factors on the RHS cancel out when~$A$ is only Hilbert--Schmidt.}
	and the constant~$\lambda \in \R$ is given in eq.~\eqref{prop:chaos_decomp_hessian:eq2} below.
	The Hilbert-Schmidt operator operator~$A$ is defined in~\eqref{eq:op_A} and characterises the second chaos component of~$\hat{Q}_\sh$ while~$\tilde{A}$ is given in~\eqref{eq:Atilde} and captures the \enquote{singular part} of~$A$.
\end{theorem}

\begin{remark} \label{rmk:assumptions}
	Let us comment on the assumptions in the previous theorem, specifically on the role of the parameter~$\kappa$.  
	Usually, one fixes a parameter~$0 < \kappa \ll 1$ to begin with such that realisations~$\xi(\omega)$ lie in the space~$\CC^{-1-\kappa}(\T^2)$ for a.e.~$\omega \in \Omega$ -- and, in turn, Theorem~\ref{thm:precise_laplace} applies with~$\tilde{\eta} = 1 - \kappa$.
	In contrast, the previous theorem requires a \emph{slightly} stronger assumption both on~$F$ (for example, Hypothesis~\ref{ass:h1} on~$F$ is stronger since~$\eta = 1 - 5\kappa < 1 - \kappa = \tilde{\eta}$) and~$u_0$ (in the form that it needs to live in the \emph{smaller} space~$\CC^{1+2\kappa}(\T^2)$);
	the reason for that is detailed in Remark~\ref{rmk:reg_u_0} below.
	However, since~$\kappa > 0$ can be chosen arbitrarily small, the marginally stronger assumption on~$F$ is not really a restriction in practice; concerning the stronger assumption on~$u_0$, see Remark~\ref{rmk:spaces_with_blowup} below.
\end{remark}

\begin{remark}[The case of \emph{linear}~PAM.] \label{rmk:linear_pam}
	We note that, since~$g(u) = u$ is not bounded, the assumption that~$g \in \CC_b^{N+7}$ in Theorem~\ref{thm:precise_laplace} and thus in Theorem~\ref{thm:main_result} \emph{formally} rules out linear PAM.
	However, both theorems remain valid in that case: 
	Since the underlying stochastic PDE is linear, the analysis only becomes easier and one even gets global existence.\footnote{Under appropriate conditions on the non-linearity~$g$, global existence for gPAM has recently been shown by Chandra, Feltes, and Weber~\cite{chandra_feltes_weber_24}.} 
	Furthermore, one can check that the regularity structures analysis in~\cite[Sec.~2.3]{friz_klose_22}, specifically~\cite[Prop.~2.39]{friz_klose_22} (the estimate on the abstract Taylor terms) and, most importantly,~\cite[Prop.~2.40]{friz_klose_22} (the estimate on the abstract Taylor remainder) remains valid in case of (affine) linear~$g$ -- in fact, the analysis becomes much easier in that case since~$g^{(k)} \equiv 0$ for~$k \geq 2$.
\end{remark}

As an ingredient for the proof of Theorem~\ref{thm:main_result}, we establish a refined version of the result advertised in~\cite[Sec.~0.$4$]{friz_klose_22} of our previous article: The minimiser~$\sh$ given in~\ref{ass:h2} lives in a space of better regularity than just~$\CH$. 

\begin{theorem} \label{thm:reg_minimiser} 
	Let~$\bar{\theta} > 0$ and take~$\kappa > 0$ to be small enough.
	In addition: 
	\begin{enumerate}[label=(\arabic*), itemsep=5pt]
		\item \label{thm:reg_minimiser:ass_1} Let~$F$ satisfy the requirements of Theorem~\ref{thm:precise_laplace} with~$\eta = 1-2\kappa-\bar{\theta}$ and~$N = 0$.\footnote{\label{fn:ass_F}In fact, the reader can check that not \emph{all} of these assumptions on~$F$ are necessary for the conclusion of Theorem~\ref{thm:reg_minimiser} to hold, see Remark~\ref{rmk:ass_F} below. 
			We have decided to present the theorem in this form so as not to clutter up the exposition.} 
		\item \label{thm:reg_minimiser:ass_2} Let the non-linearity~$g$ be as in~Theorem~\ref{thm:precise_laplace} with~$N \in \N_0$, i.e.~$g \in C_b^{N+7}(\R,\R)$.
	\end{enumerate}
	Then, for any tuple~$(\gamma,u_0)$ which satisfies
	\begin{equation} \label{eq:cond_n}
		\gamma \in [0,  N + 4 + \kappa), \quad u_0 \in \CC^{\gamma + 1 -\kappa}(\T^2)
	\end{equation}
	we have
	\begin{equation}  \label{thm:reg_minimiser:eq} 
		\sh \in H^{\theta \wedge \bar{\theta}} \quad \text{for} \quad \theta \in (\gamma,\gamma + 1 - 2\kappa) \, . 
	\end{equation}
\end{theorem}

We believe that the previous result is also of independent interest: 
Thus, we have aimed for the strongest possible result even though that level of generality is not needed for the proof of Theorem~\ref{thm:main_result}. 
In fact, the latter only requires~$\sh \in H^{3\kappa}$, a consequence of the previous theorem with~$\gamma = 0$ and~$\bar{\theta} = 3\kappa$.

\begin{remark} \label{rmk:increased_reg_assumption}
	Note that the minimum~$\theta \wedge \bar{\theta}$ in~\eqref{thm:reg_minimiser:eq} ties the improved $L^2$ Sobolev regularity of the minimiser~$\sh$ to the regularity assumption~\ref{thm:reg_minimiser:ass_1}. 
	As it becomes ever stronger as~$\bar{\theta}$ increases, this is the price one has to pay for the increased regularity of~$\sh$.
	In order to understand in detail in which way the $H^{\theta \wedge \bar{\theta}}$-norm of~$\sh$ depends on the regularity of~$F$ (and various other quantities), see ~\eqref{eq:dep_norm_h_F} as well as \eqref{prop:summary_der_eq:value_constant} and \eqref{prop:summary_der_eq:gronwall} below.
\end{remark}

\begin{remark} \label{rmk:spaces_with_blowup}
	Observe that for increasing values of~$n$, the assumption~\eqref{eq:cond_n} on~$u_0$ becomes ever more restrictive. 
	In particular, the assumptions in~\eqref{eq:cond_n} are stronger than those in Theorem~\ref{thm:precise_laplace} when~$\gamma \geq \kappa$. 
	Perhaps this situation can be avoided by working with the spaces~$\CC^{\alpha,\eta}$ with prescribed behaviour at the origin, see~\cite[Def.~$6.2.$ and Rmk.~$6.3$]{hairer_rs}. 
	However, such considerations would only increase technical complexity, so we chose not to follow this idea any further.
\end{remark}

\paragraph*{Previous work.} 
Precise asymptotics for Laplace-type functionals 
\begin{equation*}
	J(\eps) = \E\sbr[2]{\exp\del[2]{-\frac{F(X^\eps)}{\eps^{2}}}}
\end{equation*}
have been studied by many authors for different choices of stochastic processes~$X^\eps$: The long history of this endeavour has been reviewed in~\cite{friz_klose_22}. Therefore, it is not surprising that a formula for~$a_0$ such as~\eqref{thm:eq_a0} has appeared in the literature before, for example in
\begin{itemize}[itemsep=5pt]
	\item Ben Arous's work~\cite[Thm.~$4$]{ben_arous_laplace} when~$X^\eps$ is a diffusion process, i.e. the solution to a specific type of stochastic differential equation~(SDE),
	\item the articles by Inahama~\cite[Prop.~$8.1$]{inahama06} and then by Inahama and Kawabi~\cite[Thm.~$6.5$]{inahama_kawabi} when~$X^\eps$ solves a related rough differential equation~(RDE) driven by a Brownian rough path, and
	\item a follow-up work by Inahama~\cite{inahama_rde_fbm}, when the RDE is driven by a \emph{fractional} Brownian rough path with Hurst parameter~$H \in (\nicefrac{1}{4},\nicefrac{1}{2})$.
\end{itemize}
In fact, Kusuoka and Osajima~\cite{kusuoka2008remark} dedicated a whole article to the computation of the coefficient~$a_0$ in the expansion for the density of certain Wiener functionals studied by Kusuoka and Stroock~\cite{kusuoka1991precise}.

The result closest in spirit to ours is contained in the above-mentioned work~\cite{inahama_rde_fbm} by Inahama: Although not stated explicitly, the results of~ \cite[Sec.~$6.2$, Prop.~$5.2$, and~Lem.~$5.4$]{inahama_rde_fbm} contain the exact analogue to ours in the fBM-RDE setting.

\paragraph*{Our contribution.} 

In the article at hand, we will consider~$X^\eps = \hat{u}^\eps$ to be the renormalised solution to gPAM, i.e. eq.~\eqref{eq:gpam_shifted_eps} with~$h \equiv 0$.
It is precisely this need for~\emph{renormalisation} and the technical complexity of the corresponding solution theory provided by Hairer's regularity structures~\cite{hairer_rs} that renders the analysis in our case challenging.
Even so, it is striking to note that our formula for~$a_0$ is of the same type as those for the different~$X^\eps$ above when there is \emph{no renormalisation} at all.
Morally, this is in line with the findings of Hairer and Weber~\cite{hairer_weber_ldp} who proved that the renormalisation constants do not appear on the level of the large deviation rate functional.

On a more technical level, renormalisation enters our analysis via the splitting of the operator~$A$ into $q := A-\tilde{A}$ and $\tilde{A}$ and therefore sheds some light on that seemingly \enquote{ad hoc} decomposition: 
\begin{itemize}[itemsep=5pt]
	\item The operator~$\tilde{A}$ is such that~$\tilde{A}[\xi_\d,\xi_\d]$ contains all the singular products that need renormalisation when~$\d \to 0$.
	\item The operator~$q$ contains all the products such that~$q[\xi,\xi]$ makes sense even without renormalisation.\footnote{These products are still quadratic in the noise, hence this statement is non-trivial.}
\end{itemize}
In particular, this observation dovetails nicely with the existing literature: By Goodman's Theorem (see Proposition~\ref{prop:q_trace_class}), we can show that~$q$ is trace-class so that the exponent in~\eqref{thm:eq_a0} is indeed well-defined.
However, the required estimate for Goodman's Theorem to apply is only true because the \emph{minimiser~$\sh$ from}~\ref{ass:h2} \emph{has better than just Cameron--Martin regularity}! 
This fact is somewhat hidden in the above-mentioned works of Ben Arous and Inahama and Kawabi already but, to the best of our knowledge, has never been explored systematically. 
In particular, this is true for our (S)PDE context in which the spatial dependence poses additional challenges and requires marginally stronger assumptions that we have commented on in Remark~\ref{rmk:assumptions} above.

Finally, let us mention that we obtained the formula~\eqref{thm:eq:coeff:a_0} for~$a_0$ in~\cite[Sec.~$3.2$]{friz_klose_22} by a large deviation argument and~\emph{without} any reference to a formula of type~\eqref{thm:eq_a0} -- in contrast to the works of Inahama and Inahama and Kawabi referenced above.
Their argument went the opposite way and crucially relied on a direct proof that~$A$ is a Hilbert--Schmidt operator: 
In our case, this is a simple byproduct of the formula~\eqref{thm:eq:coeff:a_0}, see Theorem~\ref{prop:Qh_second_chaos} below.
This further clarifies and adds to the~\enquote{big picture} of Laplace asymptotics.

\paragraph*{Structure of the article.} 

In section~\ref{sec:improved_reg}, we use a bootstrap argument to show that the minimiser~$\sh$ has better than just Cameron--Martin regularity (Theorem~\ref{thm:reg_minimiser}). 
This is done via analysing gPAM driven by deterministic Sobolev noise (subsection~\ref{sec:gpam_sobolev}), the corresponding equation for the directional derivative~(subsection~\ref{sec:der_eq:gpam}), and, finally, by leveraging~\ref{ass:h2} via first-order optimality~(subsection~\ref{sec:first_order_opt}).
The proof of Theorem~\ref{thm:reg_minimiser} can be found on page~\pageref{pf:reg_min}.

Section~\ref{sec:wic_decomp_hessian} starts with a short review of the connection between exponential moments of elements in the second (inhomogeneous) Wiener--It\^{o} chaos and Carleman--Fredholm determinants, followed by the generic formula for~$a_0$ that this correspondence implies in our case~(proposition~\ref{prop:sec_wic_carleman}).
In the rest of this section, we compute the chaos decomposition of~$\hat{Q}_\sh$.
In subsection~\ref{sec:comp_second_chaos}, we show that the Hilbert--Schmidt operator~$A$ that characterises the second chaos component of the Hessian~\mbox{$\hat{Q}_\sh = Q_\sh(\hbz)$} is given by~$Q_\sh$ applied to canonical lifts of Cameron--Martin functions~(up to a constant, see Theorem~\ref{prop:Qh_second_chaos}).
In subsection~\ref{sec:comp_zeroth_chaos}, we define a Hilbert--Schmidt operator~$\tilde{A}$ such that improved regularity of~$\sh$ implies that $q := A - \tilde{A}$ is trace class by Goodman's Theorem. As a result, we find an explicit formula for~$\hat{Q}_\sh$ that allows to calculate its expected value as a sum of~$\operatorname{Tr}(q)$ and an explicit constant~\mbox{$\lambda \in \R$}~(see proposition~\ref{prop:chaos_decomp_hessian}).
We close that subsection with the proof of theorem~\ref{thm:main_result}, see~page~\pageref{pf:thm_1}.
In subsection~\ref{subsec:eta}, we establish an approximation result for~$\lambda$ that allows to do computations in practice.

Appendix~\ref{app:rs_extended} contains some new constructions within regularity structures and some references to our companion article~\cite{friz_klose_22} that complement Table~\ref{table:symbols} below.
Appendix~\ref{sec:aux_besov} lists auxiliary results on Sobolev and Besov spaces, in particular concerning embeddings, multiplication, and the action of the heat semigroup on those spaces.	

\paragraph*{Notational conventions.}
\begin{itemize}[itemsep=3pt]
	\item By~$a \aac b$ we mean that there exists a constant~$c \in \R$ (which depends on the parameter~$\chi$ if we write~\enquote{$\aac_\chi$}) s.t.~$a \leq c b$. In multiple such estimates, said constant may vary from line to line. 
	\item For~$n \in \N$, we write~$[n] := \{1,\ldots,n\}$.
	\item In addition, we identify the symmetric tensor product~$H^{\odot 2}$ of a Hilbert space~$H$ with the space of symmetric Hilbert--Schmidt bilinear forms (or, equivalently, symmetric Hilbert-Schmidt operators) on~$H$, see e.g.~\cite[Ex.~E.$13$]{janson}.
	\item In contrast to~\cite{friz_klose_22}, we do not colour modelled distributions w.r.t. the extended regularity structure in green because the two-fold extension in~App.~\ref{app:rs_extended} would require another colour.
	\item Unless noted otherwise, all the Besov and Sobolev spaces are considered w.r.t.~$\T^2$. Further, for a Banach space~$X$ and~$T>0$, we write~$C_T X := C([0,T];X)$ for the space of continuous functions on~$[0,T]$ with values in~$X$. For~$v \in C_T X$, we write~$v(s) := v(s,\cdot) \in X$ for~$s \in [0,T]$.
	\item Finally, by~\enquote{$A \rightsquigarrow B$ in eq.~$(15)$} we mean that we exchange the symbol~\enquote{$A$} by the symbol~\enquote{$B$} in equation~$(15)$ but otherwise keep everything the same.
\end{itemize}

\paragraph*{Glossary.}
As mentioned earlier, Section~\ref{sec:wic_decomp_hessian} is set in the framework of the theory of regularity structures; 
we summarise some of the frequently appearing symbols in the following table. 
We also provide a reference to each symbol's definition, mostly within the summary of our companion article~\cite[App.~A]{friz_klose_22}, which the reader may consult for a reference to the original source.
For some more details, including some new definitions, see Appendix~\ref{app:rs_extended}.

\vspace{-1em}
\begin{center} 
	{\small
	\renewcommand{\arraystretch}{1.4}
	\begin{longtable}{p{.08\textwidth}p{.6\textwidth}p{.23\textwidth}}
		\toprule
		\label{table:symbols}
		\centering Object & Meaning & Ref.\\
		\midrule
		\endhead
		\bottomrule
		\endfoot
		\centering $\CX_T$ & \emph{Solution space} for~\eqref{eq:gpam_shifted_eps}, i.e.~$\CX_T := C_T\CC^\eta(\T^2)$ for some \mbox{$\eta \in (\nicefrac{1}{2},1)$} & p.~\pageref{e:solution_space} \\
		\centering $\DD^{\gamma,\eta}(\bz)$ & \emph{Space of singular modelled distributions} for~$\gamma > 0$ and~$\eta \in \R$ w.r.t. a model~$\bz$ which, depending on the context, is an element of~$\MM$, $\MM[\<cm>]$, or~$\MM[\<cm>,\<cm2>]$ & \protect\cite[Def.~A.31]{friz_klose_22}\\
		\centering $E_h$ & \emph{Extension operator}; for~$h \in \CH$, we have~$E_h: \MM \to \MM[\<cm>]$, $\bz \mapsto E_h \bz$ & \protect\cite[Prop.~A.13]{friz_klose_22}\\
		\centering $\tilde{E}_h$ & --- ; for~$h \in \CH$, we have~$\tilde{E}_h: \MM[\<cm>] \to \MM[\<cm>,\<cm2>], \ \bz \mapsto \tilde{E}_h \bz$ & p.~\pageref{def:amended_extension_op} \\
		\centering $E_{(h,k)}$ & --- ; for~$h,k \in \CH$, we have~$E_{(h,k)}: \MM \to \MM[\<cm>,\<cm2>], \ \bz \mapsto E_{(h,k)} \bz$ & p.~\pageref{def:extension_operator}\\
		\centering $\Phi$ & \emph{Abstract solution map} associated to eq.~\eqref{eq:gpam_shifted_eps}; we have $\Phi(\bz)~=~\CR^{\bz}(\CS(\bz))$ and, in particular,~$\hat{u}^\eps_h := \Phi(T_h \hbz)$ & $\CS$ is given in \protect\cite[Eq.~(2.19)]{friz_klose_22} \\
		\centering $\CH$ & \emph{Cameron--Martin space} of 2D spatial white noise~(SWN)~$\xi$, i.e. $\CH = L^2(\T^2)$ & p.~\pageref{thm:precise_laplace} \\
		\centering $\LL$ & \emph{Canonical lifting operator}, which associates to any~$h \in \CH$ a canonical lift~$\LL(h) \in \MM$ in model space & \protect\cite[Def.~A.12]{friz_klose_22}\\
		\centering $\MM$ & \emph{Space of admissible models} w.r.t. to the regularity structure~$\TT$ for gPAM & \protect\cite[Def.~A.6]{friz_klose_22} \\
		\centering $\MM[\<cm>]$ & --- w.r.t. to the \emph{extended} regularity structures~$\TT[\<cm>]$ for gPAM & \textcolor{cameron}{$\MM$}~in~\protect\cite[Def.~A.6]{friz_klose_22} \\
		\centering $\MM[\<cm>,\<cm2>]$ & --- w.r.t. to the \emph{(2-fold) extended} regularity structure~$\TT[\<cm>,\<cm2>]$ for gPAM & p.~\pageref{app:extended_models} \\
		\centering $\CP$ & \emph{Abstract convolution with the heat kernel}; we write~$\CP^{\operatorname{e}_h} := \CP^{E_h \bz}$ and $\CP^{\operatorname{e}_{(h,k)}} := \CP^{E_{(h,k)} \bz}$ if~$\bz \in \MM$ is clear from context & \protect\cite[Sec.~3.7]{cfg} \\
		\centering $\boldsymbol{P}$ & \emph{Polynomial model} & \protect\cite[Sec.~13.3.1]{friz-hairer} \\
		\centering $\CR$ & \emph{Reconstruction operator}; $\CR^{\bz}: \DD^{\gamma,\eta}(\bz) \to \CX_T$ & \protect\cite[Prop.~6.9]{hairer_rs} \\ 
		\centering $T_h$ & \emph{Translation operator}; for~$h \in \CH$, we have~$T_h: \MM \to \MM$, \ $\bz \mapsto T_h \bz$ & \protect\cite[Prop.~A.14]{friz_klose_22}\\
		\centering $\tilde{T}_h$ & --- \emph{for extended models}; for~$h \in \CH$, we have~$T_h: \MM[\<cm>] \to \MM[\<cm>]$, $\bz \mapsto \tilde{T}_h \bz$ & p.~\pageref{def:amended_extension_op} \\
		\centering $\bz$ & Generic element of~$\MM$, $\MM[\<cm>]$, or~$\MM[\<cm>,\<cm2>]$, called an (extended)~\emph{model} & -- \\
		\centering $\hbz^{\xi_\delta}, \hbz$ & \emph{BPHZ} model associated with the mollified SWN~$\xi_\delta$ and its limit as~$\delta \to 0$ & \protect\cite[Sec.~A.2.2]{friz_klose_22} \\
		\caption{Frequently occurring symbols and references to their definitions}
	\end{longtable}
	}
\end{center}

\vspace{-3em}
\paragraph*{Acknowledgements.} 
The author thanks Carlo Bellingeri, Ilya Chevyrev, Peter Friz, and Paul Gassiat for helpful discussions on the content presented in this article.
He also wishes to thank Nils Berglund for various discussions on the subject during the Masterclass and Workshop of the Graduate School on \enquote{Higher Structures Emerging from Renormalisation} (8 to 19 November 2021) at the Erwin Schrödinger International Institute for Mathematics and Physics (ESI) of the University of Vienna.
Further thanks are due to Lucas Broux and David Lee for the clarification on Besov multiplication that lead to Remark~\ref{rmk:besov_multiplication}.
Finally, the author is grateful to the anonymous referee for pointing out a mistake in an earlier version of this article: The process of fixing it lead to a great improvement of the article overall.

\paragraph*{Funding.}
Financial support from the DFG through the International Research Training Group (IRTG) 2544 \enquote{Stochastic Analysis in Interaction} is gratefully acknowledged.
This article has been written while the author was employed at TU Berlin and revised when he was employed at the University of Warwick.

\section{Improved regularity of the minimiser} \label{sec:improved_reg}

We start by analysing the regularity of the minimiser~$\sh$ from~\ref{ass:h2} via pure PDE arguments. 
At the end of this section, we give a proof of Theorem~\ref{thm:reg_minimiser}, see p.~\pageref{pf:reg_min} below.
This result only enters the proof of Theorem~\ref{thm:main_result} in the form of Corollary~\ref{lemma:Nzero} via Proposition~\ref{prop:q_trace_class}. 
Thus, readers only interested in the representation formula for~$a_0$ can directly continue with Section~\ref{sec:wic_decomp_hessian}.

The following subsections develop the arguments we need in order to prove Theorem~\ref{thm:reg_minimiser}.

\subsection{Deterministic gPAM with Sobolev noise} \label{sec:gpam_sobolev}

We begin by considering deterministic gPAM driven by a noise in the~$L^2$-based Sobolev space~$H^\gamma(\T^2)$.
At first, we need a technical lemma about composition operators in the spaces~$\CC^\a(\T^2)$ for~$\a \notin \N$:

\begin{lemma} \label{lem:comp_sobolev} 
	Let~$\a \in (0,\infty) \setminus \N$ and~$\nu \in (\a,\infty)$. Given a function~$g \in \CC^{\nu+1}(\R)$, the composition operator
	\begin{equation*}
		\fC_g: \CC^\a(\T^2) \to \CC^\a(\T^2), \quad \fC_g(f) := g \circ f
	\end{equation*}
	is well-defined and locally Lipschitz continuous.
\end{lemma}
We present two proofs for this lemma: 
One which specialises a general statement about Besov spaces from~\cite{runst_sickel} to the case at hand and one which specialises the general results in~\cite{hairer_rs} to the simple case of the \emph{polynomial regularity structure}.
We emphasise that one can translate the second proof into a language that does not require regularity structures at all – but its framework is convenient for us in order to reference the specific results which we employ.
In addition, it provides a simpler first instance of the more general statements which we are going to encounter in Section~\ref{sec:wic_decomp_hessian}.

Recall the definition of the spaces~$\CC^\alpha(\T^2)$ provided in Appendix~\ref{sec:aux_besov}. 
One can check that, for~$\alpha = n + \gamma$ with~$n \in \N_0$ and~$\gamma \in (0,1)$, one can equivalently define~$\CC^{n+\gamma}(\T^2)$ as the space of~$n$ times continuously differentiable functions \mbox{$f : \T^2 \to \R$} such that their derivatives of order~$n$ are $\gamma$-Hölder continuous.

\begin{proof}[Proof \textnormal{(via textbook reference)}]
	Let us first assume that~$g(0) = 0$. 
	For~$\a \in (0,1)$, the statement is obvious. 
	Let now $\a > 1$.
	Note that for our choice of~$\a$, we have~$B_{\infty,\infty}^\a \equiv \CC^\a$ and thus~$\CC^\a \embed L^\infty$ such that the claim is a special case of~\cite[Thm.~$5.5.2/1$]{runst_sickel}. 
	Note that this result actually assumes~$g' \in \CC^\infty(\R)$ but~\cite[Prop.~$5.5.1/1$]{runst_sickel} shows that one can weaken this assumption to the one we claimed given that~$\nu > 0$.
	
	In case~$g(0) \neq 0$, we define~$\tilde{g} := g - g(0)$ and note that the claim is true for~$\fC_{\tilde{g}}$ by what we have already proved.
	We further note that
	%
	$\fC_{\tilde{g}}(f) 
	= 
	\fC_{g}(f) - g(0)$ 	
	and, accordingly,
	\begin{equation*} \label{aux_comp_op}
		\fC_{\tilde{g}}(f)(x) - \fC_{\tilde{g}}(f)(y) =  \fC_{g}(f)(x) - \fC_{g}(f)(y), \quad
		\fC_{\tilde{g}}(f_1) - \fC_{\tilde{g}}(f_2)
		=
		\fC_{g}(f_1) - \fC_{g}(f_2).
	\end{equation*}
	In case~$\alpha \in (0,1)$, a direct computation based on the previous observations shows that the claim is also true for~$\fC_g$. 
	For~$\alpha > 1$, observe that~$\tilde{g}^{(k)} = g^{(k)}$ for all~$k \in \N_{\geq 1}$;
	the claim for~$\fC_g$ then follows from the corresponding statement for~$\fC_{\tilde{g}}$ by definition of the spaces~$\CC^\alpha$. 
\end{proof}

The following proof builds upon the \emph{polynomial regularity structure}~$\bar{\CT}$ and the \emph{polynomial model}~$\boldsymbol{P}$; 
for a pedagogical introduction to these notions (and regularity structures more generally), see~\cite[Sec.~13.2.1]{friz-hairer} and \cite[Sec.~13.3.1]{friz-hairer}.
We further remind the reader that all the other symbols appearing in the proof are listed in Table~\ref{table:symbols}, incl. references to their definitions.

\begin{proof}[Proof \textnormal{(via regularity structures)}]
	As before, we write $\alpha = n + \gamma$ for~$n \in \N_0$ and~$\gamma \in (0,1)$ and recall the content of~\cite[Prop.~13.16]{friz-hairer}:
	\begin{enumerate}[label=(\roman*), itemsep=5pt]
		\item \label{e:composition_rs_i} Let~$f \in \CC^\alpha$. Defining 
		\begin{equation} \label{e:composition_rs_1}
			F(x) := f(x)\mathbf{1} + \sum_{1 \leq \abs[0]{k} \leq n} \frac{f^{(k)}(x)}{k!} \boldsymbol{X}^k
		\end{equation}
		one then has~$F \in \DD^\alpha(\boldsymbol{P})$.
		\item \label{e:composition_rs_ii} Conversely, if~$\tilde{F} \in \DD^\alpha(\boldsymbol{P})$, then~$f := \scal{\tilde{F},\mathbf{1}}$ is in~$\CC^{\alpha}$ and~$\tilde{F} = F$ where the latter is defined from~$f$ via~\eqref{e:composition_rs_1}.
	\end{enumerate}
	We now let~$f \in \CC^\alpha$, define~$F$ as in~\eqref{e:composition_rs_1}, and also introduce the \emph{abstract composition operator}
	\begin{equation*}
		\fC_G(F)(x) := (G \circ F)(x) := \sum_{\ell \geq 0} \frac{1}{\ell!} g^{(\ell)}(f(x))\del[1]{F(x) - f(x)\mathbf{1}}^{\ell} \, .
	\end{equation*}
	We note the following:
	\begin{itemize}[itemsep=5pt]
		\item By definition, $F(x) - f(x)\mathbf{1} \in T_{\geq 1}$ since the polynomials~$X_i$ for~$i = 1,2$ are of degree~$1$.
		\item The polynomial regularity structure is a function-like sector (see~\cite[Def.~2.5]{hairer_rs}) within itself.
	\end{itemize}
	As a consequence of the first observation and our choice of~$\nu$,~\cite[Prop.~14.8]{friz-hairer} shows that~$\fC_G$ is locally Lipschitz con\-ti\-nuous from~$\DD^\alpha(\boldsymbol{P})$ into itself.
	Combined with the second observation and~\ref{e:composition_rs_ii} above, this implies that
	\begin{equation} \label{e:reconstruction}
		\CR^{\boldsymbol{P}}(F) = \scal{F,\mathbf{1}} = f \in \CC^\alpha, \qquad
		\CR^{\boldsymbol{P}}(\fC_G(F)) = \scal{\fC_G(F),\boldsymbol{1}} = \fC_g(f) \in \CC^\alpha
	\end{equation}
	where, in both cases, the first identities in~\eqref{e:reconstruction} are due to~\cite[Prop.~3.28]{hairer_rs}.
	
	Finally, observe that the $\DD^\alpha(\boldsymbol{P})$-norm of~$F$ defined from~$f \in \CC^\alpha$ via~\eqref{e:composition_rs_1} and the $\CC^\alpha$-norm of~$f$ itself are equivalent.
	Since the reconstruction operator~$\CR^{\boldsymbol{P}}$ is also locally Lipschitz continuous by~\cite[Thm.~3.10]{hairer_rs}, we then find
	\begin{equation*}
		\norm[0]{\fC_g(f_1) - \fC_g(f_2)}_{\CC^\alpha}
		=
		\norm[0]{\CR^{\boldsymbol{P}}(\fC_G(F_1)) - \CR^{\boldsymbol{P}}(\fC_G(F_2))}_{\CC^\alpha}
		\lesssim_M
		\norm[0]{F_1 - F_2}_{\DD^\alpha(\boldsymbol{P})}
		\lesssim
		\norm[0]{f_1 - f_2}_{\CC^\alpha}
	\end{equation*}
	which holds uniformly over all~$f_1,f_2 \in \CC^\alpha$ with~$\norm[0]{f_1}_{\CC^\alpha}, \norm[0]{f_2}_{\CC^\alpha} \leq M$ 
	and, for~$i = 1,2$, $F_i$ is defined from~$f_i$ via~\eqref{e:composition_rs_1}. 
\end{proof}

\begin{remark}
	As is noted in~\cite[Rmk.~$5.5.2/1$]{runst_sickel}, \emph{global} Lipschitz continuity of~$\fC_g$ is not to be expected 
	unless~$g(x) = cx$ for some~$c \in \R$. This, of course, would correspond to the \enquote{ordinary} (rather than \enquote{generalised}) PAM driven by a SWN with covariance~$c^2 \d(x-y)$, cf. also Remark~\ref{rmk:linear_pam}.
\end{remark}

The following proposition is a straight-forward adaptation of classical arguments, see for example~\cite[Thm.~$4.3.4$]{cazenave_haraux}. We spell them out for the sake of completeness; see also~\cite[Prop.~A.1]{chouk_friz} for the case~$\gamma = 0$.

\begin{proposition}\label{prop:ex_det_gpam_Hgamma}
	Let $\kappa > 0$,~$N \in \N_0$, and~$\gamma \geq 0$ such that~$\gamma  + 1 - \kappa \notin \N$ and $\gamma < N + 5 + \kappa$.
	Then, for~$g \in C_{b}^{N+7}(\R)$, 
	$u_0 \in \CC^{\gamma+1-\kappa}(\T^2)$, and $h \in H^\gamma(\T^2)$, there exists a maximal time~$T_\star = T_\star(u_0,\norm[0]{h}_{H^{\gamma}},\gamma) > 0$ such that for each~$T < T_\star$, the solution to the equation
	\begin{equation} \label{prop:ex_det_gpam_Hgamma:eq}
		(\partial_t - \Delta) w_h = g(w_h)h, \quad w_h(0,\cdot) = u_0,
	\end{equation}
	satisfies~$w_h \in C_T \CC^{\gamma + 1 - \kappa}(\T^2)$.
\end{proposition}

\begin{proof}
	Let~$X := \CC^{\gamma+1-\kappa}(\T^2)$ and note that the assumptions on~$\gamma$ are precisely such that the conditions from Lemma~\ref{lem:comp_sobolev} are met with~$\alpha := \gamma + 1 - \kappa$ and~$\nu := N + 6$.
	In mild formulation, equation~\eqref{prop:ex_det_gpam_Hgamma:eq} reads
	\begin{equation}
		w_h(t) = P_t u_0 + \int_0^t P_{t-s} \sbr[1]{g(w_h(s)) h} \dif s, 
		\label{eq:w_h:mild_form}
	\end{equation}
	where $P_t = e^{t\Delta}$ denotes the heat semigroup and~$t \in [0,T]$ for some~$T > 0$, the existence of which is yet to be established. 
	For proving~\emph{(local) existence} of~$w_h$, let~$M \geq \norm[0]{u_0}_X$ and define
	\begin{equation*}
		B_{T,M} := \{u \in C_T X: \ \norm[0]{u}_{C_T X} \leq M\}
	\end{equation*}
	to be the closed ball of radius~$M$ in~$C_T X$. For~$v \in B_{T,2M}$, we further define
	\begin{equation*}
		\Gamma_T(v)(t) := P_t u_0 + \int_0^t P_{t-s}\sbr[0]{g(v(s))h} \dif s, \quad t \in [0,T].
	\end{equation*}
	By Lemma~\ref{lem:comp_sobolev} and standard multiplication theorems in Besov spaces~(see Proposition~\ref{prop:mult_besov}) we have	
	\begin{equation}
		g(v(s)) \in X, \qquad
		g(v(s)) h \in H^{\gamma} 
		\embed \CC^{\gamma-1-\nicefrac{\kappa}{2}},
		\qquad
		\norm[0]{g(v(s))h}_{\CC^{\gamma - 1-\nicefrac{\kappa}{2}}} \leq C \norm[0]{g(v(s))}_X \norm[0]{h}_{H^{\gamma}}
		\label{prop:ex_det_gpam_Hgamma:pf_aux1}
	\end{equation}	
	for all~$s \in [0,T]$ and some~$C > 0$.
	Next, note that we have~$g(v(s)) = g(v(s)) - g(0)$ and thus
	\begin{equation*}
		\norm[0]{g(v(s))}_{X} \leq L (2M) 2M ,
	\end{equation*}
	where~$L(2M)$ denotes the local Lipschitz constant from Lemma~\ref{lem:comp_sobolev}.
	The regularising effect of the heat semigroup~(see Proposition~\ref{prop:heat_besov}) then allows to obtain the estimate
	\begin{align*}
		\norm[0]{\Gamma_T(v)(t)}_X
		& \leq
		\norm[0]{u_0}_X
		+
		\int_0^t \norm[0]{P_{t-s} \sbr[0]{g(v(s)) h}}_X \dif s 
		\leq 
		M
		+
		\int_0^t (t-s)^{\frac{1}{2}(-2+\frac{\kappa}{2})} \norm[0]{g(v(s)) h}_{\CC^{\gamma - 1-\nicefrac{\kappa}{2}}} \dif s \\
		& \leq 
		M + C \frac{4}{\kappa} T^{\frac{\kappa}{4}} \norm[0]{g(v(s))}_{X} \norm[0]{h}_{H^\gamma} 
		\leq 
		M + C \frac{4}{\kappa} T^{\frac{\kappa}{4}} L(2M) \norm[0]{h}_{H^\gamma} 2M.
	\end{align*}
	We then choose~$T= T_M > 0$ such that
	\begin{equation}
		C \frac{4}{\kappa} T_M^{\frac{\kappa}{4}}  L(2M) \norm[0]{h}_{H^\gamma} \leq \frac{1}{2}
		\label{eq:choice_time}
	\end{equation}
	which implies that~$\Gamma_{T_M}$ maps~$B_{T_M,2M}$ into itself.
	Similarly, for~$v,w \in B_{T,2M}$ we find
	\begin{equation*}
		\norm[0]{\Gamma_{T_M}(w)(t) - \Gamma_{T_M}(v)(t)}_{X}
		\leq
		C \frac{4}{\kappa}  T_M^{\frac{\kappa}{4}} L(2M) \norm[0]{h}_{H^\gamma} \norm[0]{w-v}_{\CC_T X} 
		\leq
		\frac{1}{2} \norm[0]{w-v}_{\CC_T X}.
	\end{equation*}
	This shows that~$\Gamma_{T_M}$ is a contraction in~$B_{T_M,2M}$ and, combined with the previous result, therefore admits a unique fixed-point. 
	As a consequence, we know that
	\begin{equation*}
		T_\star \equiv T_{\star}(u_0,\norm[0]{\sh}_{H^{\gamma}},\gamma) := \sup\{T > 0: \eqref{eq:w_h:mild_form} \ \text{admits a solution in} \in \CC_T X \}
	\end{equation*}
	satisfies~$T_\star > 0$.\footnote{Note that~$T_\star$ depends on~$\gamma$ via~$X$ and on~$\norm[0]{\sh}_{H^\gamma}$ via the estimate~\eqref{eq:choice_time}.} We can construct a maximal solution~$w_h \in \CC([0,T_\star);X)$ by iterating the procedure above since \emph{solutions are unique}, as we prove now.
	
	Assume there exists some~$T>0$ such that~$v,w \in \CC_T X$ both solve~\eqref{eq:w_h:mild_form}. Then we define
	\begin{equation*}
		M := \sup_{t \in [0,T]} \del[1]{\thinspace\norm[0]{v(t)}_X \vee	\norm[0]{w(t)}_X}		
	\end{equation*}
	so that~$v,w \in B_{T,M}$. As before, we then have the estimate
	\begin{equation*}
		\norm[0]{v(t) - w(t)}_X 
		\leq 
		C  \norm{h}_{H^\gamma}  L(M) \int_0^t (t-s)^{\frac{1}{2}(-2 + \frac{\kappa}{2})} \norm[0]{v(s) - w(s)}_X \dif s
	\end{equation*}
	which implies~$v = w$ by Gronwall's inequality.
\end{proof}

\begin{remark} \label{rmk:existence_time}
	In what follows, we will always assume that~$t \in [0,T]$ and~$s \in [0,t]$ where~$T < T_\star$ .
	In the proof of Theorem~\ref{thm:reg_minimiser} presented on p.~\pageref{pf:reg_min} below, each induction step will see~$\sh \in H^{\gamma}(\T^2)$ gain regularity, that is,~$\gamma$ will successively get larger. 
	As a consequence,~$T_\star$ will potentially shrink as it is easily seen via~\eqref{eq:choice_time} to be a decreasing function of~$\gamma$. 
	However, we can only ever take \emph{finitely many} induction steps due to the constraints on~$\gamma$ in~\eqref{eq:cond_n} or the constraint imposed by~$\bar{\theta}$ in Theorem~\ref{thm:reg_minimiser}, assumption~\ref{thm:reg_minimiser:ass_1}; 
	as a consequence, we will always have~$T_\star > 0$.
	
	In addition, we emphasise that the equations for~$v_{\sh,k}$ in~\eqref{eq:vhk_mild} as well as that for~$v_{\sh,k,l}^{(\<cms>)}$ in~\eqref{e:vhkl_cm} (which we will analyse in Propositions~\ref{prop:ex_det_gpam_Hgamma} resp.~\ref{prop:q_trace_class} below) are \emph{linear} equations:
	As a result, they will also exist up to time~$T < T_\star$ since~$w_{\sh}$ enters their respective fixed-point equations as an input.  
\end{remark}

\subsection{The derivative equation} \label{sec:der_eq:gpam}

In~\cite[Thm.~$2$]{friz_klose_22}, we proved\footnote{In fact, we established a Taylor expansion for~$\hat{u}_h^\eps$ with \emph{estimates} on both the Taylor terms and, crucially, its remainder.} that the solution~$\hat{u}^\eps_h$ to~\eqref{eq:gpam_shifted_eps} is of class~$\CC^{N+3}$ in the parameter~\mbox{$\eps \ll 1$}. For~$i \in [N+3]$ and~\mbox{$\bz \in \MM$}, we may thus set
\begin{equation*}
	u_h^{(i)}(\bz) := \partial_\eps^{i}\sVert[0]_{\eps = 0} u_h^\eps(\bz), \quad
	\hat{u}_h^{(i)} := u_h^{(i)}(\hbz).
\end{equation*}
In this subsection, we study the directional derivative
\begin{equation} \label{e:def_vhk}
	v_{h,k} := \partial_\eps\sVert[0]_{\eps = 0} w_{h + \eps k}, \quad h, k \in \CH \,;
\end{equation}
it can easily be seen to satisfy the equality
\begin{equation} \label{der_eq:comparison}
	v_{h,k} = u_h^{(1)}(\LL(k)).
\end{equation}
which justifies that~$v_{h,k}$ exists in the first place.
In addition, observe that the latter satisfies the deterministic PDE
\begin{equation} \label{eq:vhk_pde}
	(\partial_t - \Delta) v_{h,k} = g'(w_h)v_{h,k}h + g(w_h)k, \quad v_{h,k}(0,\cdot) = 0
\end{equation}
which, for~$t \in [0,T]$, reads in mild formulation as follows:
\begin{equation} \label{eq:vhk_mild}
	v_{h,k}(t) = \int_0^t P_{t-s} \sbr[1]{g'(w_h(s)) v_{h,k}(s) h} \dif s + \int_0^t P_{t-s} \sbr[1]{g(w_h(s))k} \dif s
	=: I_{h,k}(t) + J_{h,k}(t)
\end{equation}

\begin{remark} \label{rmk:extension_v_h_k}
	As we shall see in Proposition~\ref{prop:summary_der_eq} below, the PDE~\eqref{eq:vhk_pde} and its corresponding mild formulation~\eqref{eq:vhk_mild} are \emph{not} restricted to the case~$k \in \CH = L^2(\T^2)$.
	In fact, if we fix~$h \in H^\gamma$ for~$\gamma \geq 0$, we may use this observation to \emph{extend} our earlier definition of~$v_{h,k}$ in~\eqref{e:def_vhk} to~$k \in H^{-\theta}$ where the value of~$\theta$ is constrained by that of~$\gamma$, see eq.~\eqref{e:choice_theta} below.
	Note that this does not require~$w_{h+\eps k}$ to be well-defined for~$h \in H^\gamma$, $k \in H^{-\theta}$, and~$\eps > 0$.
\end{remark}

Suppose now that $h \in H^\gamma$ for $\gamma \geq 0$ and let $k \in H^{-\theta} = \del[0]{H^\theta}^*$ for some $\theta > \gamma$; eligible choices for~$\theta$ are presented in~\eqref{e:choice_theta} below.
Our goal is to quantify the regularity of~$v_{h,k}$ and prove an estimate of type~$\norm[0]{v_{h,k}}_Y \aac \norm[0]{k}_{H^{-\theta}}$ for a suitable space of continuous functions~$Y = Y(\theta)$, see Proposition~\ref{prop:summary_der_eq} below for a precise statement.

This estimate will then be used in Subsection~\ref{sec:first_order_opt} below: 
In the specific case where~$h = \sh \in H^{\gamma}$ is the minimiser from Hypothesis~\ref{ass:h2}, we will combine it with additional arguments to show that~$\sh \in H^{\theta}$ as long as~$\theta \leq \bar{\theta}$ where the latter is as in Theorem~\ref{thm:reg_minimiser}.
In other words: The minimiser~$\sh$ gains~$\theta - \gamma > 0$ units in $L^2$ Sobolev regularity; the constraint that~$\theta \leq \bar{\theta}$ is explained in Remark~\ref{rmk:increased_reg_assumption} above.

We present the main result of this subsection:

\begin{proposition} \label{prop:summary_der_eq}
Let~$T, \kappa, N, \gamma, g,$ and~$u_0$ as in Proposition~\ref{prop:ex_det_gpam_Hgamma}. 
In addition, assume that 
the more restrictive bound $\gamma < N + 4 + \kappa$ is satisfied. 
Now let~$h \in H^\gamma$ and recall the definition of~$\bar{\theta} > 0$ in Theorem~\ref{thm:reg_minimiser}, assumption~\ref{thm:reg_minimiser:ass_1}, above. 
Then, for any~
\begin{equation} \label{e:choice_theta}
	\theta \in I(\gamma)
	:=
	(\gamma, \gamma + 1 - 2\kappa)
\end{equation}
and~$k \in H^{-\theta}$, we then have~$v_{h,k} \in C_T  \CC^{-\theta +1-2\kappa}$ as well as the estimate
\begin{equation}
	\norm[0]{v_{h,k}}_{C_T \CC^{-\theta +1-2\kappa}}
	\aac \norm[0]{k}_{H^{-\theta}}
	\label{prop:summary_der_eq:estimate}
\end{equation}
where the implicit constant~$\tilde{C}$ depends on~$g,g',h,w_h,T,\kappa,\gamma,$ and~$\theta$;
its precise value is given in eq.~\eqref{prop:summary_der_eq:gronwall} below.
\end{proposition}

\begin{proof}
The pattern of proof is similar to that of Proposition~\ref{prop:ex_det_gpam_Hgamma}: 
One shows the invariance of suitable balls and then establishes the contraction property of the fixed-point map associated with the mild formulation~\eqref{eq:vhk_mild}. 
Since it is a \emph{linear} equation, the only factor that constrains the existence time is the external input~$w_h$ given in~\eqref{prop:ex_det_gpam_Hgamma:eq}, hence our choice of the time horizon~$T$ in the statement of the proposition.

Since the fixed-point argument (which includes the continuity in time) is quite standard, we only analyse the products in~\eqref{eq:vhk_mild} to find the correct \emph{spatial} regularity of the solution~$v_{h,k}$.
To this end, we analyse the terms $I_{h,k}$ and $J_{h,k}$ in~\eqref{eq:vhk_mild} separately, still working in the setting of Proposition~\ref{prop:ex_det_gpam_Hgamma}, albeit with~$\gamma < N + 4 + \kappa$. 
The latter guarantees that Lemma~\ref{lem:comp_sobolev} is applicable to~$g' \in C_{b}^{N+6}(\R)$ with~$\alpha := \gamma + 1 - \kappa$ and~$\nu := N+5$.

\paragraph*{--- The inhomogeneous term $J_{h,k}$.}
We analyse the product~$g(w_h(s))k$ in~\eqref{eq:vhk_mild} first. Since~$k \in H^{-\theta}$ by assumption and $g(w_h(s)) \in \CC^{\gamma + 1 - \kappa}$ by Proposition~\ref{prop:ex_det_gpam_Hgamma} and Lemma~\ref{lem:comp_sobolev} above, we require 
$\gamma + 1 - \kappa - \theta > 0$
for the product
\begin{equation}
	g(w_h(s))k \in H^{-\theta} \embed \CC^{-\theta - 1 -\kappa}
	\label{eq:prod_inhomog_term}
\end{equation}
to be well-defined. Therefore, since we also want to have~$\theta > \gamma$, we pick up the constraint
\begin{equation} \label{e:theta_1}
	\theta \in (\gamma, \gamma + 1 -\kappa)
\end{equation} 
Using the smoothing effect of the heat semigroup $P_t$ again, we obtain the estimate
\begin{equs}[][eq:aux_1]
	\thinspace &
	\norm[0]{J_{h,k}(t)}_{\CC^{-\theta + 1 - 2\kappa}} 
	= \,
	\norm[3]{\int_0^t P_{t-s} [g(w_h(s))k] \dif s}_{\CC^{-\theta + 1 - 2\kappa}} 
	\aac
	\int_0^t \norm[0]{P_{t-s} [g(w_h(s))k]}_{\CC^{-\theta + 1 - 2\kappa}} \dif s \notag \\
	\aac \thinspace &
	\int_0^t (t-s)^{\frac{1}{2}(-2 + \kappa)} \norm[0]{g(w_h(s))k}_{\CC^{-\theta - 1 - \kappa}} \dif s 
	\aac
	\frac{2}{\kappa} T^{\frac{\kappa}{2}}
	\norm[0]{g(w_h)}_{\CC_T \CC^{\gamma + 1 -\kappa}} \norm[0]{k}_{H^{-\theta}}
\end{equs}
which implies that, \emph{at best}, we have 
\begin{equation} \label{e:ass_reg_vhk}
	v_{h,k} \in  C_T \CC^{-\theta + 1 - 2\kappa}
\end{equation}
which is our standing assumption going forward.

\paragraph*{--- The homogeneous term $I_{h,k}$.}
%
Proposition~\ref{prop:ex_det_gpam_Hgamma} and Lemma~\ref{lem:comp_sobolev} imply that 
\begin{equation*}
	g'(w_h(s)) \in \CC^{\gamma + 1 - \kappa}
\end{equation*}
and thus, recalling that $-\theta + 1 - 2\kappa < \gamma + 1 - \kappa$ due to~\eqref{e:theta_1}, the product
\begin{equation} \label{e:first_prod}
	g'(w_h(s)) v_{h,k}(s) \in \CC^{-\theta + 1 - 2\kappa}
\end{equation}
is well-defined provided
\begin{equation*}
	0 < -\theta + 1 - 2\kappa + \gamma + 1 - \kappa \quad \Longleftrightarrow \quad \theta < \gamma + 2 - 3\kappa.
\end{equation*}
This, in turn, is implied by the even more restrictive constraint in~\eqref{e:theta_1}.
For the product of~\eqref{e:first_prod} with~$h \in H^\gamma$ to be well-defined, we require that
\begin{equation*}
	\del[0]{-\theta + 1 - 2\kappa} + \gamma > 0 \quad \Longleftrightarrow \quad \theta < \gamma + 1 - 2\kappa
\end{equation*}
which slightly strengthens the constraint in~\eqref{e:theta_1} to
\begin{equation}\label{e:theta_2}
	\theta \in (\gamma, \gamma + 1 -2\kappa)
\end{equation}
We now distinguish two scenarios:
\begin{enumerate}[label=(\roman*), itemsep=5pt]
	\item \textbf{Scenario}~$\boldsymbol{1}$: Assume that~$\gamma \leq - \theta + 1 - 2\kappa$. 
	In that case, we have
	\begin{equation*}
		g'(w_h(s)) v_{h,k}(s)h \in H^{\gamma} \embed \CC^{\gamma -1 - \kappa}
	\end{equation*}
	and with~$\beta := \gamma - 1 - \kappa - (-\theta + 1 - 2\kappa)$, we further find
	\begin{equs}[eq:reg_min_term_B_case1]
		\norm[0]{I_{h,k}(t)}_{\CC^{-\theta + 1 -2\kappa}}
		& = \,
		\norm[3]{\int_0^t P_{t-s} [g'(w_h(s)) v_{h,k}(s)h] \dif s}_{\CC^{-\theta + 1 -2\kappa}}  \\
		& \lesssim \,
		\int_0^t (t-s)^{\frac{1}{2}\beta} \norm[0]{g'(w_h(s)) v_{h,k}(s)h}_{\CC^{\gamma - 1 - \kappa}} \dif s \\
		& \aac \, 
		\norm{g'(w_h)}_{C_T \CC^{\gamma + 1 -\kappa}} \norm[0]{h}_{H^\gamma}
		\int_0^t (t-s)^{\frac{1}{2}\beta} \norm[0]{v_{h,k}(s)}_{\CC^{-\theta + 1 -2\kappa}} \dif s.
	\end{equs}
	In this computation, we have assumed that $\beta \in (-2,0]$ which, in equivalent form, reads \mbox{$\theta \in (-\gamma-\kappa,2-\gamma-\kappa]$}.
	Combining this condition with the assumption $\gamma \leq - \theta + 1 - 2\kappa$ and with~\eqref{e:theta_2} we arrive at
	\begin{equs}
		\theta & \in (\gamma, \gamma + 1 -2\kappa) \cap (-\gamma-\kappa,2-\gamma-\kappa] \cap (0, -\gamma + 1 - 2\kappa] \\
		& = (\gamma, \gamma + 1 -2\kappa) \cap (0, -\gamma + 1 - 2\kappa]
	\end{equs}
	One can now easily see that this intersection of intervals merits the following case distinction:
	\begin{itemize}
		\item If~$\gamma = 0$, we arrive at the condition~$\theta \in (0,1-2\kappa) = (\gamma,-\gamma+1-2\kappa)$, with open left and right interval boundaries.
		\item If~$\gamma > 0$, then~$\theta \in (\gamma,-\gamma + 1 -2\kappa]$, with open left but closed right boundary. 
		In order not to get an empty interval in that case, we require
		\begin{equation*}
			\gamma < -\gamma + 1 - 2\kappa \quad \Longleftrightarrow \quad \gamma < \frac{1}{2} - \kappa
		\end{equation*}
	\end{itemize}
	
	Altogether, we summarise this scenario as 
	\begin{equs}[][e:theta_scenario_1]
		\gamma = 0 \quad &  \rightsquigarrow \quad \theta \in (0,1-2\kappa) = (\gamma,\gamma+1-2\kappa) \, , \notag\\
		\gamma \in \del[0]{0,\nicefrac{1}{2}-\kappa}
		\quad &  \rightsquigarrow \quad
		\theta \in (\gamma,-\gamma + 1 - 2\kappa] \, , 
	\end{equs}
	where the intervals on the respective right hand sides detail eligible choices for~$\theta$.
	\item \textbf{Scenario}~$\boldsymbol{2}$: Assume that~$\gamma > - \theta + 1 - 2\kappa$. 
	In that case, the product
	\begin{equation*}
		g'(w_h(s)) v_{h,k}(s)h \in H^{-\theta +1 - 2\kappa} \embed \CC^{-\theta - 3\kappa}.
	\end{equation*}
	Similarly as before, we get
	\begin{equs}[eq:reg_min_term_B_case2]
		\norm[0]{I_{h,k}(t)}_{\CC^{-\theta + 1 -2\kappa}}
		& = \,
		\norm[3]{\int_0^t P_{t-s} [g'(w_h(s)) v_{h,k}(s)h] \dif s}_{\CC^{-\theta + 1 -2\kappa}} \\
		& \lesssim \,
		\int_0^t (t-s)^{\frac{1}{2}(-1-\kappa)} \norm[0]{g'(w_h(s)) v_{h,k}(s)h}_{\CC^{-\theta - 3\kappa}} \dif s \notag\\
		& \aac \, 
		\norm{g'(w_h)}_{C_T \CC^{\gamma + 1 -\kappa}} \norm[0]{h}_{H^\gamma}
		\int_0^t (t-s)^{\frac{1}{2}(-1-\kappa)} \norm[0]{v_{h,k}(s)}_{\CC^{-\theta + 1 -2\kappa}} \dif s.
	\end{equs}
	In this scenario, combining the assumption that~$\gamma > - \theta + 1 - 2\kappa$ and~\eqref{e:theta_2} we arrive at the condition
	\begin{equation} \label{e:theta_scenario_2}
		\theta \in \del[1]{\gamma \vee (-\gamma + 1 - 2\kappa),\gamma + 1 - 2\kappa}
	\end{equation}
	specifying eligible choices for~$\theta$ in this Scenario.
	Let us have a closer look at the maximum on the LHS of this interval:
	\begin{itemize}[itemsep=5pt]
		\item The condition~$\gamma < -\gamma +1 - 2\kappa$ is equivalent to~$\gamma < \nicefrac{1}{2} - \kappa$, just as in Scenario~$1$ above. 
		However, for the resulting interval~$(-\gamma + 1 - 2\kappa,\gamma +1 - 2\kappa)$ in~\eqref{e:theta_scenario_2} to be non-empty, we find that this requires~$\gamma > 0$.
		In summary, we see that
		\begin{equation} \label{e:theta_scenario_2_case_1}
			\gamma \in (0, \nicefrac{1}{2} - \kappa) \quad \rightsquigarrow \quad
			\theta 
			\in
			\del[0]{-\gamma + 1 - 2\kappa,\gamma +1 - 2\kappa} \, .
		\end{equation}
		\item The case~$\gamma \geq -\gamma +1 - 2\kappa$ is then equivalent to~$\gamma \geq \nicefrac{1}{2} - \kappa$.
		In that case, no further conditions are required for the corresponding interval in~\eqref{e:theta_scenario_2} to be non-empty.
		In summary, we have
		\begin{equation} \label{e:theta_scenario_2_case_2}
			\gamma \geq \nicefrac{1}{2} - \kappa \quad  \rightsquigarrow \quad
			\theta 
			\in
			\del[0]{\gamma,\gamma +1 - 2\kappa} \, .
		\end{equation}
	\end{itemize}
\end{enumerate}
We emphasise that the intervals in~\eqref{e:theta_scenario_1} and~\eqref{e:theta_scenario_2}  which~$\theta$ can be chosen from are \emph{disjoint}, as they must, because the conditions that define Scenarios~$1$ and~$2$ are mutually exclusive.
In case~$\gamma \in (0,\nicefrac{1}{2}-\kappa)$, however, note that one can combine the conditions in~\eqref{e:theta_scenario_1} and~\eqref{e:theta_scenario_2_case_1} to give the single condition
\begin{equation} \label{e:theta_scenarios_1_2_combined}
	\gamma \in (0, \nicefrac{1}{2} - \kappa) \quad \rightsquigarrow \quad \del[0]{\gamma, \gamma +1 - 2\kappa}
\end{equation}
In conjunction with~\eqref{e:theta_scenario_1} (in case~$\gamma = 0$) and \eqref{e:theta_scenario_2_case_2} (in case~$\gamma \geq \nicefrac{1}{2} - \kappa$), this explains our choice of~$I(\gamma)$ in~\eqref{e:choice_theta}.
\paragraph*{--- Gronwall's inequality.}	
Combining the fixed-point equation for~$v_{h,k}$, eq.~\eqref{eq:vhk_mild}, together with 
\begin{itemize}[itemsep=5pt]
	\item the estimate on~$J_{h,k}$ in~\eqref{eq:aux_1} 
	\item the estimate on~$I_{h,k}$ in~\eqref{eq:reg_min_term_B_case1} resp.~\eqref{eq:reg_min_term_B_case2}, depending on the values of~$\gamma$ and~$\theta$,
\end{itemize}
we then arrive at the estimate
\begin{equs}[][eq:aux_2]
	\norm[0]{v_{h,k}(t)}_{\CC^{-\theta + 1 -2\kappa}}
	& \lesssim
	\frac{2}{\kappa} T^{\frac{\kappa}{2}}
	\norm[0]{g(w_h)}_{C_T \CC^{\gamma + 1 -\kappa}} \norm[0]{k}_{H^{-\theta}} \\
	& \quad +
	\norm[0]{g'(w_h)}_{C_T \CC^{\gamma + 1 -\kappa}} \norm[0]{h}_{H^\gamma}
	\int_0^t (t-s)^{\frac{1}{2}\beta} \norm[0]{v_{h,k}(s)}_{\CC^{-\theta + 1 -2\kappa}} \dif s \\
	& = 
	C\del[3]{
		\norm[0]{k}_{H^{-\theta}}
		+
		\int_0^t (t-s)^{\frac{1}{2}\beta} \norm[0]{v_{h,k}(s)}_{\CC^{-\theta + 1 -2\kappa}} \dif s}
\end{equs}
where
\begin{equation*}
	\beta 
	=
	\begin{cases}
		\gamma - 1 - \kappa - (-\theta + 1 - 2\kappa) & \quad \text{in Scenario~1}  \\
		-1-\kappa & \quad \text{in Scenario~2}
	\end{cases}
\end{equation*}
and
\begin{equation} \label{prop:summary_der_eq:value_constant}
	C = C(T,g,g',h,w_h,\kappa,\gamma,\theta)
	:=
	\frac{2}{\kappa} T^{\frac{\kappa}{2}}
	\norm[0]{g(w_h)}_{C_T \CC^{\gamma + 1 -\kappa}}
	+
	\norm[0]{g'(w_h)}_{C_T \CC^{\gamma + 1 -\kappa}} \norm[0]{h}_{H^\gamma}
\end{equation}
Gronwall's inequality applied to~\eqref{eq:aux_2} now implies the estimate
\begin{equation} \label{prop:summary_der_eq:gronwall}
	\norm[0]{v_{h,k}(t)}_{\CC^{-\theta + 1 -2\kappa}}
	\leq C  \exp\del[3]{C \int_0^t (t-s)^{\frac{1}{2}\beta} \dif s} \norm[0]{k}_{H^{-\theta}}
	= \tilde{C} \norm[0]{k}_{H^{-\theta}}
\end{equation}
where the quantity~$\tilde{C}$ is defined in the obvious way via~$C$ and therefore depends on the same parameters.
The constant~$\tilde{C}$ is finite because we have~$\beta > -2$, both in Case~$1$ and~$2$, by choice of the respective parameter regime for~$\theta$.
\end{proof}

\subsection{Leveraging first-order optimality}	\label{sec:first_order_opt}

Up to now, we have considered \emph{generic} elements~$h \in H^\gamma$ for some~$\gamma \geq 0$; 
in this subsection, in contrast, we consider the \emph{specific} choice~$h = \sh$ from Hypothesis~\ref{ass:h2} above and leverage the fact that it \emph{minimises} the functional~$\FF$.

Recall that, a priori,~$\sh \in \CH = H^0$.
Using~\eqref{der_eq:comparison} above, a straight-forward calculation shows that the first-order optimality condition for~$\sh$ to be a minimiser of~$\FF$, namely~$D \FF\sVert[0]_{\sh} = 0$, is equivalent to
\begin{equation} \label{eq:f_o_o}
\scal{\sh,k}_\CH = DF\sVert[0]_{w_\sh}\del[0]{v_{\sh,k}}, \quad k \in \CH.
\end{equation}
Note that both~$\sh$ and~$k$ are assumed to be elements of~$\CH$ in the previous equality.
However, as we have commented on in Remark~\ref{rmk:extension_v_h_k} above, $v_{\sh,k}$ does make sense if~$\sh \in H^\gamma$ and~$k \in H^{-\theta}$ for $\gamma \geq 0$ and \emph{some} values of~$\theta > \gamma$.
In fact, we have proved in Proposition~\ref{prop:summary_der_eq} above that 
\begin{equation*}
v_{\sh,\bullet} \in \CL(H^{-\theta},C_T \CC^{1 - 2\kappa-\theta })
\end{equation*}
for any admissible choice of~$\theta = \theta(\gamma)$ as given in~\eqref{e:choice_theta}.
Therefore, the RHS of~\eqref{eq:f_o_o} still makes sense if one can plug elements of~$C_T\CC^{1 -2\kappa - \theta}$ into~$DF\sVert[0]_{w_\sh}$ -- which, in turn, is the raison d'\^{e}tre for Theorem~\ref{thm:reg_minimiser}, assumption~\ref{thm:reg_minimiser:ass_1}.
We direct the reader to Remark~\ref{rmk:increased_reg_assumption} to recall the role of~$\bar{\theta}$ within that assumption.

Collecting all of these observations, we can finally give a proof of Theorem~\ref{thm:reg_minimiser}.

\begin{proof} \label{pf:reg_min}
Let~$\gamma$ be as in~\eqref{eq:cond_n}. 
In this proof, we will inductively construct a sequence~$(\gamma_n)_n$ in the following way:
\begin{enumerate}[label=(\arabic*),itemsep=5pt, start=0]
	\item \label{thm:reg_minimiser:pf_item0} 
	We set~$\gamma_0 := 0$ in which case we know \emph{a priori} that~$\sh \in \CH \equiv H^{\gamma_0}$.
	\item \label{thm:reg_minimiser:pf_item1} 
	For some~$n \in \N_0$, consider some~$0 \leq \gamma_n \leq \gamma$ for which~$\sh \in H^{\gamma_n}$ and observe that the assumptions for Propositions~\ref{prop:ex_det_gpam_Hgamma} and~\ref{prop:summary_der_eq} are satisfied for~$\gamma_n$.
	\item \label{thm:reg_minimiser:pf_item2} Let~$\theta \in (\gamma_n,\gamma_n+1-2\kappa)$ and~$k \in \CC^{\infty}(\T^2)$.
	By Proposition~\ref{prop:summary_der_eq} with~$h = \sh$, we then have
	\begin{equation} \label{eq:dep_norm_h_F}
		\abs[0]{\scal{\sh,k}_\CH}
		=
		\abs[0]{DF\sVert[0]_{w_\sh}(v_{\sh,k})}
		\aac
		\norm[0]{DF\sVert[0]_{w_\sh}}_{\CL(C_T\CC^{1-2\kappa-\theta\wedge \bar{\theta}},\R)} \norm[0]{v_{\sh,k}}_{C_T \CC^{1 - 2\kappa-(\theta \wedge \bar{\theta})}}
		\aac
		\norm[0]{k}_{H^{-(\theta \wedge \bar{\theta})}}.
	\end{equation}
	where the last estimate is due to~\eqref{prop:summary_der_eq:estimate}. 
	\vspace{2pt}
	
	Note that we have taken the minimum~$\theta \wedge \bar{\theta}$ in the previous estimate to comply with Theorem~\ref{thm:reg_minimiser}, assumption~\ref{thm:reg_minimiser:ass_1}, i.e. such that~$DF\sVert[0]_{w_\sh}(v_{\sh,k})$ is always well-defined. 
	In this regard, note that
	\begin{equation*}
		1 - 2 \kappa -\theta \wedge \bar{\theta} \geq 1 - 2 \kappa -\theta 
		\quad \implies \quad
		C_T \CC^{1 - 2 \kappa -\theta \wedge \bar{\theta}}(\T^2) 
		\embed
		C_T \CC^{1 - 2 \kappa - \iota}(\T^2), \quad \iota \in \{\theta,\bar{\theta}\} \, .
	\end{equation*}
	In particular, for~$\iota = \bar{\theta}$, this implies that
	\begin{equation*}
		DF\sVert[0]_{w_\sh} \in \CL\del[1]{C_T \CC^{1 - 2 \kappa - \bar{\theta}}(\T^2), \R} 
		\quad \implies \quad
		DF\sVert[0]_{w_\sh} \in \CL\del[1]{C_T \CC^{1 - 2 \kappa - \theta \wedge \bar{\theta}}(\T^2), \R} 
	\end{equation*}
	where the LHS of the previous implication is true by Theorem~\ref{thm:reg_minimiser}, assumption~\ref{thm:reg_minimiser:ass_1}.
	\vspace{2pt}
	
	However, note as well that if~$\iota = \theta$ and~$\theta \wedge \bar{\theta} = \bar{\theta}$, the conclusion of Proposition~\ref{prop:summary_der_eq} that~$v_{\sh,k} \in C_T \CC^{1 - 2 \kappa -\bar{\theta}}(\T^2)$ is only valid if~$\bar{\theta} \in (\gamma_n,\gamma_n+1-2\kappa)$ itself.
	This, in turn, is ensured by the termination conditions for the algorithm, see Steps~\ref{thm:reg_minimiser:pf_item3} and~\ref{thm:reg_minimiser:pf_item4} below.
	
	We conclude the previous arguments as follows:
	\begin{equs}
		\norm[0]{\sh}_{H^{\theta \wedge \bar{\theta}}}
		& =
		\sup\left\{ \frac{\abs[0]{\scal{\sh,k}_{H^{\theta \wedge \bar{\theta}} \x H^{-{\theta \wedge \bar{\theta}}}}}}{\norm[0]{k}_{H^{-{\theta \wedge \bar{\theta}}}}}: k \in H^{-{\theta \wedge \bar{\theta}}}(\T^2) \right\} \\
		& =
		\sup\left\{ \frac{\abs[0]{\scal{\sh,k}_{\CH}}}{\norm[0]{k}_{H^{-{\theta \wedge \bar{\theta}}}}}: k \in \CC^\infty(\T^2) \right\}
		< \infty
	\end{equs}
	where we have used~\cite[Prop.~1.58]{bcd} in the last equality.
	As a consequence, we find~$\sh \in H^{\theta \wedge \bar{\theta}}$.
	\item \label{thm:reg_minimiser:pf_item3} 
	On the one hand, the claim in~\eqref{thm:reg_minimiser:eq} follows if~$\gamma_n = \gamma$, in which case the proof is finished.
	On the other hand, the claim also follows if~$\gamma_n + 1 -2\kappa > \bar{\theta}$ because in that case, there exists~$\theta \in (\gamma_n,\gamma_n+1-2\kappa)$ for which~$\theta \wedge \bar{\theta} = \bar{\theta}$.
	\item \label{thm:reg_minimiser:pf_item4} 
	If neither of the two conditions in~\ref{thm:reg_minimiser:pf_item3} are met, let~$\gamma_{n+1} := (\gamma_n + 1 - 3 \kappa) \wedge \gamma$ and repeat Steps~\ref{thm:reg_minimiser:pf_item1} to \ref{thm:reg_minimiser:pf_item4} with~$\gamma_{n+1}$ instead of~$\gamma_n$ until one of the termination conditions is true.
	
	Note, in particular, that his choice of~$\gamma_{n+1}$ entails that that, if $\gamma_n + 1 - 2\kappa > \bar{\theta}$, then \mbox{$\bar{\theta} \in (\gamma_n,\gamma_n+1-2\kappa)$}.
\end{enumerate}
\end{proof}

\begin{remark} \label{rmk:ass_F}
As indicated in footnote~\ref{fn:ass_F} on p.~\pageref{fn:ass_F}, the previous proof reveals that not \emph{all} of the assumptions on~$F$ postulated in Theorem~\ref{thm:reg_minimiser}, assumption~\ref{thm:reg_minimiser:ass_1}, are necessary.
More precisely, the assumption that~$F \in C_b^1\del[0]{C_T\CC^{1-2\kappa-\bar{\theta}};\R}$ in conjunction with Hypothesis~\ref{ass:h2} would be sufficient; clearly, these conditions are implied by Theorem~\ref{thm:reg_minimiser}, assumption~\ref{thm:reg_minimiser:ass_1}.
\end{remark}

The following corollary records some of the previous results when they are specialised to the setting of Theorem~\ref{thm:main_result}. 
We will refer to it in the proof of Proposition~\ref{prop:q_trace_class} below.

\begin{corollary} \label{lemma:Nzero}
Consider the assumptions of Theorem~\ref{thm:main_result}.
We then have~$\sh \in H^{3\kappa}(\T^2)$ as well as~$w_\sh \in C_T \CC^{1+2\kappa}(\T^2)$ and~$g^{(m)}(w_\sh) \in C_T \CC^{1+2\kappa}(\T^2)$ for all~$T < \bar{T}_\star := T_\star(u_0,\norm[0]{\sh}_{H^{3\kappa}},3\kappa)$.
\end{corollary}

\begin{proof}
The regularity statement for~$\sh$ is the content of Theorem~\ref{thm:reg_minimiser} with~$\gamma = 0$ and~$\theta = \bar{\theta} =3\kappa$.\footnote{Note that we only require~$u_0 \in \CC^{1-\kappa}(\T^2)$ rather than the stronger assumption~$u_0 \in \CC^{1+2\kappa}(\T^2)$ from Theorem~\ref{thm:main_result} for this to be true. The latter is only necessary for the statements about~$w_\sh$ and~$g^{(m)}(w_\sh)$ which enter the proof of Proposition~\ref{prop:q_trace_class} below, cf. Remark~\ref{rmk:reg_u_0}.} 
The second assertion is the content of Proposition~\ref{prop:ex_det_gpam_Hgamma} with~$\gamma_\star := 3\kappa$ and the last claim follows from Lemma~\ref{lem:comp_sobolev}.
\end{proof}

In line with the comments in Remark~\ref{rmk:existence_time}, we will always take~$T > 0$ as in the previous corollary from now on.

\section{Wiener chaos decomposition of the Hessian} \label{sec:wic_decomp_hessian}

Let~$0 < \kappa \ll 1$ and~$(B,\CH,\mu)$ with~$B := \CC^{-1-\kappa}(\T^2)$,~$\CH := L^2(\T^2)$, and~$\mu = \operatorname{Law}(\xi)$ for $2$D SWN~$\xi$. It is well-known that this is an~\emph{abstract Wiener space} in the sense of Gross~\cite{gross}.

With~$u_\sh^{(i)}(\bz)$ and~$\hat{u}_\sh^{(i)} $ introduced in the beginning of subsection~\ref{sec:der_eq:gpam}, the~\emph{Hessian}~$\hat{Q}_\sh$ is given by
\begin{equation*}
	Q_\sh(\bz) := \partial_{\eps}^{2}\sVert[0]_{\eps = 0} F\del[1]{u_\sh^\eps(\bz)}, \quad
	\hat{Q}_\sh := Q_\sh(\hbz)
\end{equation*}
or, in explicit form, as 
\begin{equation}
	\hat{Q}_\sh 
	= D^2F\sVert[0]_{w_\sh}\sbr[1]{\hat{u}_\sh^{(1)},\hat{u}_\sh^{(1)}}
	+
	DF\sVert[0]_{w_\sh}\sbr[1]{\hat{u}_\sh^{(2)}}.
	\label{eq:Q_h_repeated}
\end{equation}
The following lemma provides a first generic formula for~$a_0$ using Wiener chaos theory.  

\begin{proposition} \label{prop:sec_wic_carleman}
	The random variable~$\hat{Q}_\sh$ is an element of~$\CH_2 \oplus \CH_0$, where~$\CH_k$ is the $k$-th order homogeneous Wiener-It\^{o} chaos of~$\xi$. 
	As such, it admits a representation
	\begin{equation}
		\hat{Q}_\sh 
		= 
		\E\sbr[1]{\hat{Q}_\sh} 
		+ 
		I_2(C), 
		\qquad
		C = \sum_{k=1}^\infty \lambda_k e_k \otimes e_k
		\label{eq:decomp_Qh}
	\end{equation} 
	where~$(e_k)_{k \in \N}$ is an ONB of the CM space~$\CH$ and~$C$ a symmetric Hilbert-Schmidt operator, i.e.~$(\lambda_k)_{k \in \N} \in \ell^2(\R)$. In addition, the following assertions are true:
	\begin{enumerate}[label=(\roman*)]
		\item The coefficient~$a_0$ in~\eqref{thm:eq:coeff:a_0} admits the representation
		\begin{equation}
			a_0 
			= e^{-\frac{1}{2} \E\sbr[0]{\hat{Q}_\sh}} \sbr[4]{\prod_{k=1}^\infty (1+\lambda_k) e^{-\lambda_k}}^{-\frac{1}{2}}.
			\label{eq:expr_a0_prod}
		\end{equation}
		\item The numbers~$\lambda_k$ (except possibly~$0$) are the eigenvalues (counted with multiplicities) of the compact, symmetric Hilbert--Schmidt bilinear form
		\begin{equation*}
			B(\xi(e_i),\xi(e_j)) := \frac{1}{2} \E\sbr[1]{\del[1]{\hat{Q}_\sh - \E\sbr[1]{\hat{Q}_\sh}} \xi(e_i)\xi(e_j)}, \quad i,j \in \N.
		\end{equation*}
		In addition, we have~$B(\xi(k), \xi(\ell)) = C(k,\ell)$ for all~$k,\ell \in \CH$.
	\end{enumerate}
\end{proposition}	

\begin{proof}
	Recall from~\eqref{thm:eq:coeff:a_0} that~$a_0 = \E\sbr[0]{\exp(-\nicefrac{1}{2} \, \hat{Q}_\sh)} < \infty$.
	We only need to prove the fact that \mbox{$\hat{Q}_\sh \in \CH_2 \oplus \CH_0$}: The other claims follow from standard Wiener chaos theory, see~\cite[Thm.~$6.2$]{janson}.
	
	At the level of the mollified noise~$\xi_\d$, we write~$\hat{u}^{(i)}_{\sh;\d} := u_\sh^{(i)}\del[1]{\hbz^{\xi_\d}}$, $i=1,2$. These quantities satisfy the following linear stochastic PDEs, see~\cite[Sec.~$2.5$]{friz_klose_22}:
	\begin{align}
		(\partial_t - \Delta) \hat{u}^{(2)}_{\sh;\d} 
		& = 
		2 g'(w_{\sh})u^{(1)}_{\sh;\d}  \xi_\d 
		- 2g'(w_{\sh})g(w_{\sh}) \fc_{\d} 
		+ \sbr[1]{g''(w_\sh)\del[1]{\hat{u}^{(1)}_{\sh;\d}}^2 	
			+ \hat{u}^{(2)}_{\sh;\d} g'(w_{\sh}) 
		}\sh \label{eq:u_hd1}\\
		(\partial_t - \Delta) \hat{u}^{(1)}_{\sh;\d}
		& =
		g(w_{\sh})  \xi_\d 
		+ \hat{u}^{(1)}_{\sh;\d} g'(w_{\sh}) \sh \label{eq:u_hd2}
	\end{align}
	From equation~\eqref{eq:u_hd2}, it can easily be seen that the term~$\hat{u}^{(1)}_{\sh;\d}$ is in~$\CH_1$ because it is a \emph{linear} functional of~$\xi_\d$. 
	Thus, $\hat{u}^{(2)}_{\sh;\d}$ only has components in~$\CH_k$ for~$k \in \{0,2\}$ because the RHS of~\eqref{eq:u_hd1} only contains products of two elements in~$\CH_1$ and constants.
	Since~$DF\sVert[0]_{w_\sh}$ and $D^2F\sVert[0]_{w_\sh}$ are symmetric, (bi-)linear functionals, we thus have~$\hat{Q}_{\sh;\d} := Q_\sh(\hbz^{\xi_\d}) \in \CH_2 \oplus \CH_0$.
	
	Since~$\hbz^{\xi_\d}$ converges to~$\hbz$ in prob. in~$\MM$, by continuity of~$Q_\sh$ w.r.t. the model, we have
	\begin{equation*}
		\hat{Q}_{\sh;\d} \equiv Q_{\sh}(\hbz^{\xi_\d}) \to Q_{\sh}(\hbz) \equiv \hat{Q}_\sh \qquad \text{in prob. in } \R.
	\end{equation*}
	The claim follows because Wiener chaoses are closed under convergence in probability.
	Finally, we have
	\begin{equation*}
		B(\xi(k),\xi(\ell)) 
		= \frac{1}{2} \E\sbr[1]{I_2(C)\xi(k)\xi(\ell)}
		= \frac{1}{2} \E\sbr[1]{I_2(C) I_2(k \tp \ell)}
		= C(k,\ell),
	\end{equation*}
	see, for example,~\cite[p.$9$, pt.~(iii)]{nualart}.
\end{proof}

Using the previous proposition, we can instantly rewrite the expression in~\eqref{eq:expr_a0_prod} in terms of Carleman--Fredholm determinants, a well-known generalisation of determinants to Hilbert--Schmidt operators. The following corollary is a direct consequence of~\cite[Thm.~$9.2$]{simon_trace_ideals}.

\begin{corollary} \label{coro:general}
	The formula~$a_0 =  e^{-\frac{1}{2} \E\sbr[0]{\hat{Q}_\sh}} \operatorname{det}_2(\operatorname{Id}_\CH + C)^{-\nicefrac{1}{2}}$ holds.
\end{corollary}

It remains to obtain a more explicit representation of the Wiener chaos decomposition in~\eqref{eq:decomp_Qh}. In other words, we want to see what the operator~$C$ actually is and what~$\E\sbr[0]{\hat{Q}_\sh}$ calculates to. We start with the former problem
and denote the projection of a RV onto the $k$-th Wiener-It\^{o} chaos~$\CH_k$ by~$\pi_k$.

\subsection{The component in the second chaos} \label{sec:comp_second_chaos}

In this subsection, we want to identify the operator~$C$ in~\eqref{eq:decomp_Qh}, so that we understand the part of the Hessian~$\hat{Q}_\sh$ that lies in the second homogeneous chaos~$\CH_2$.	

To this end, with~$\phi := \Phi \circ \LL$, we define the symmetric bilinear form~$A_\sh: \CH^{\tp 2} \to \R$ by:

\begin{equation}
	A_\sh[k,\ell]
	:=
	DF\sVert[0]_{w_\sh}\sbr[1]{v_{\sh,k,\ell}}
	+
	D^2 F\sVert[0]_{w_\sh}\sbr[1]{v_{\sh,k}, v_{\sh,\ell}}
	\label{eq:op_A}
\end{equation}	
where~$v_{\sh,k} := D \phi\sVert[0]_{\sh} (k)$ as well as~$v_{\sh,k,\ell} := D^2 \phi\sVert[0]_{\sh} [k,\ell]$. These derivatives satisfy the following equations in the mild formulation:
\begin{align}
	v_{\sh,k}(t) & 
	= \int_0^t P_{t-s} \sbr[1]{g'(w_\sh(s)) v_{\sh,k}(s) \sh} \dif s
	+ \int_0^t P_{t-s} \sbr[1]{g(w_\sh(s)) k} \dif s \label{eq:vhk}\\
	v_{\sh,k,\ell}(t)
	& =
	\int_0^t P_{t-s}\sbr[1]{g'(w_\sh(s)) v_{\sh,k,\ell}(s) \sh} \dif s 
	+ 
	\int_0^t P_{t-s}\sbr[1]{g''(w_\sh(s)) v_{\sh,k}(s) v_{\sh,\ell}(s) \sh} \dif s \label{eq:vhkl:1}\\
	&
	+ 
	\int_0^t P_{t-s}\sbr[1]{g'(w_\sh(s)) v_{\sh,k}(s) \ell
		+ g'(w_\sh(s)) v_{\sh,\ell}(s) k} \dif s 
	\label{eq:vhkl:2}
\end{align}

\begin{remark} \label{rmk:rel_der}
	Let~$k \in \CH$.
	As remarked in~\eqref{der_eq:comparison}, we have~$v_{\sh,k} = u_\sh^{(1)}(\LL(k))$; in addition, the identity~\mbox{$v_{\sh,k,k} = u_\sh^{(2)}(\LL(k))$} holds. As a consequence, observe that~$A_\sh[k,k] = Q_\sh(\LL(k))$.
\end{remark}

The main result in this subsection is the following:

\begin{theorem} \label{prop:Qh_second_chaos}
	The operator~$C$ coincides with~$A_\sh$. In particular,~$A_\sh$ is Hilbert-Schmidt.
\end{theorem}

Since both~$A_\sh$ and~$C$ are symmetric bilinear forms, by polarisation, it suffices to prove that they agree on the diagonal.
The theorem is a direct consequence of propositions~\ref{prop:Qh_second_chaos_aux2} and~\ref{prop:Qh_second_chaos_aux1} that follow.

\begin{proposition} \label{prop:Qh_second_chaos_aux2}
	For all~$k \in \CH$, we have
	\begin{equation}
		\partial^2_{\rho}\sVert[0]_{\rho = 0} \E\sbr[1]{Q_\sh\del[1]{\hbz(\cdot + \rho k)}}
		=
		2 A_\sh(k,k)
		\label{prop:Qh_second_chaos_aux2:eq1}.
	\end{equation}
\end{proposition}
Note that the LHS without the expectation is nothing but the second Malliavin derivative~$\nabla_k^2 \hat{Q}_\sh$ of~$\hat{Q}_\sh$ in the direction~$k \in \CH$.
In order to prove the proposition, we need three auxiliary lemmas set in the regularity structures framework for gPAM, see~\cite[Sec.~$3$]{cfg} and~\cite[App.~A]{friz_klose_22} for details. 
However, beside~$\<cm>$ representing~$\sh$, we need another CM symbol~$\<cm2>$ representing~$k$.
This is made precise in App.~\ref{app:rs_extended} below which also contains the definition of the corresponding~\emph{vector extension operator}~$E_{(\sh,k)}$.

\begin{lemma} \label{lem:aux0}
	Let~$\bz \in \MM$. For~$\gamma > 1$ and~$\eta = 1-\kappa$ as before, the fixed-point equation\footnote{By~$\DD^{\gamma,\eta}$, we denote the space of~\emph{singular modelled distributions} as introduced in~\cite[Def.~$6.2$]{hairer_rs}.}
	\begin{equation}
		U^\eps_\rho = \CP^{E_{(\sh,k)}\hbz} \del[1]{G(U^\eps_\rho) [\eps (\<wn> + \rho\<cm2>) + \<cm>]}  + \TT_\gamma Pu_0 \qquad \text{in} \quad \DD^{\gamma,\eta}\del[1]{E_{(\sh,k)} \bz} 
		\label{eq:fp_Uepskappa}
	\end{equation}
	admits a unique solution~$U^\eps_\rho$ which is~$\CC^2$ in both~$\eps$ and~$\rho$ for~$0 < \eps,\rho \ll 1$. We set
	\begin{equation*}
		U_\rho^{(i)} := \partial_\eps^{i}\sVert[0]_{\eps = 0} U^\eps_\rho, \quad
		U^{(i,j)} := \partial_\rho^{j}\sVert[0]_{\rho = 0} U^{(i)}_\rho.
	\end{equation*}
\end{lemma}
The lemma follows from standard arguments which we only sketch.
\begin{proof}
	One can easily check that the family~$(F^{\eps,\rho;E_{(\sh,k)}\bz}: \eps,\rho \in [0,1], \bz \in \MM)$ given by 
	\begin{equation*}
		F^{\eps,\rho;E_{(\sh,k)}\bz}(Y) := G(Y) [\eps (\<wn> + \rho\<cm2>) + \<cm>]
	\end{equation*}
	is uniformly strongly locally Lipschitz continuous (in~$\eps$ and~$\rho$) in the sense of~\cite[Def.~$2.6$]{friz_klose_22}. As a simple consequence of~\cite[Thm.~$7.8$]{hairer_rs}, the FP problem~\eqref{eq:fp_Uepskappa} therefore admits a unique solution. 
	For the differentiability, one uses the Implicit Function Theorem similarly to~\cite[Lem.'s~$2.21$~\&~$2.24$ and~Thm.~$2.26$]{friz_klose_22}.
\end{proof}

For the two lemmas that follow, recall that~$U^{(i)} = U^{(i)}(\bz)$ is given by~$U^{(i)} := \partial_\eps^i\sVert[0]_{\eps = 0} U^\eps(\bz)$ where~$U^\eps$ solves the FP eq. like~\eqref{eq:fp_Uepskappa} without the term~\enquote{$\rho\<cm2>$} and where~$E_{(\sh,k)}\bz$ is replaced by~$E_\sh\bz$, see~\cite[Sec.~$2.1$ \& $2.3$]{friz_klose_22} for details.

\begin{lemma} \label{lem:aux1}
	Let~$i \in \{1,2\}$. For every~$h, k \in \CH$, all~$\eps,\rho \in \R$, and every~$\bz \in \MM$ we have
	\begin{equation}
		\CR^{E_{(h,k)} \bz} U^\eps_\rho
		=
		\CR^{E_h T_{\rho k} \bz} U^\eps
		\label{lem:aux1:eq1}
	\end{equation}
	as well as
	\begin{equation}
		\CR^{E_{(h,k)} \bz} U^{(i)}_\rho
		=
		\CR^{E_h T_{\rho k} \bz} U^{(i)}.
		\label{lem:aux1:eq2}
	\end{equation}
\end{lemma}

\begin{proof}
	The claim in~\eqref{lem:aux1:eq2} follows from~\eqref{lem:aux1:eq1} and Lemma~\ref{lem:aux0} by linearity of the reconstruction operator~$\CR$.
	
	Recall the definition of~$\ft_{\<cm2s>}$ given in App.~\ref{app:rs_extended}.
	The strategy for proving~\eqref{lem:aux1:eq1} is similar to~\cite[Prop.~$2.13$]{friz_klose_22}: If we knew that
	\begin{equation}
		\ft_{\rho\<cm2>} \sbr[0]{U^\eps(E_h T_{\rho k} \bz)} = U_\rho^\eps(E_{(h,k)}\bz),
		\label{lem:aux1:pf_aux}
	\end{equation}
	then this would imply that
	\begin{equation*}
		\CR^{E_{(h,k)} \bz} U^\eps_\rho
		\overset{\eqref{lem:aux1:pf_aux}}{=} 
		\CR^{E_{(h,k)} \bz} \ft_{\rho\<cm2>} U^\eps
		= 
		\CR^{\tilde{E}_k E_h \bz}  \ft_{\rho\<cm2>} U^\eps
		=
		\CR^{\tilde{T}_{\rho k} E_h \bz} U^\eps
		= 
		\CR^{E_h T_{\rho k}\bz} U^\eps
	\end{equation*}
	where the second and the last equality are due to Lemma~\ref{lem:comm_ext_trans} and
	the third is a straight-forward modification of \mbox{\cite[Lem.~$3.22$]{cfg}}. 
	This is what we wanted to prove, so we focus on~\eqref{lem:aux1:pf_aux}.
	Observe that
	\begin{align*}
		\CP^{E_{(h,k)}\bz}\del[1]{G(\ft_{\rho \<cm2>} U^\eps) \sbr[0]{\eps(\<wn> + \rho \<cm2>) + \<cm>}}
		& =
		\CP^{\tilde{E}_k E_{h} \bz}\del[1]{\ft_{\rho \<cm2>} \del[1]{G(U^\eps) \sbr[0]{\eps\<wn> + \<cm>}}} \\
		& =
		\ft_{\rho \<cm2>} \CP^{\tilde{T}_{\rho k}E_{h} \bz}\del[1]{G(U^\eps) \sbr[0]{\eps\<wn> + \<cm>}} \\
		& =
		\ft_{\rho \<cm2>} \CP^{E_h T_{\rho k} \bz}\del[1]{G(U^\eps) \sbr[0]{\eps\<wn> + \<cm>}}
	\end{align*}
	where we have again used the same two lemmas from before. By uniqueness of~$U_\rho^\eps$ as a solution to the FP problem~\eqref{eq:fp_Uepskappa}, the claim in~\eqref{lem:aux1:pf_aux} follows.
\end{proof}

\begin{lemma} \label{lem:aux2}
	Let~$i \in \{1,2\}$. For~$\mu$-a.e.~$\omega \in B$ and all~$k \in \CH$, we have
	\begin{equation*}
		\CR^{E_{(\sh,k)} \hbz(\omega)} U^{(i,i)}
		=
		i \CR^{E_\sh \LL(k)} U^{(i)}
		\equiv
		i u_\sh^{(i)}(\LL(k)).
		\label{lem:aux1:eq3}
	\end{equation*}
\end{lemma}

\begin{proof}
	At first, we want to compute the FP eq. for~$U_\rho^{(i)}$. To this end, we can simply revisit the formulas of~\cite[Coro.~2.28]{friz_klose_22} and replace~$\<wn> \rightsquigarrow \<wn> + \rho \<cm2>$, $U^{(j)} \rightsquigarrow U^{(j)}_\rho$ for~$j \leq i$, and~$E_\sh \rightsquigarrow E_{(\sh,k)}$.
	Applying~$\partial^{i}_{\rho}\sVert[0]_{\rho = 0}$ to the resulting expression then gives~$U^{(i,i)}$.
	\begin{itemize}
		\item Let~$i=1$. Doing the replacements just described, we have
		\begin{equation*}
			U^{(1)}_{\rho} = 
			\CP^{\operatorname{e}_{(\sh,k)}}\del[1]{G'(W) U^{(1)}_{\rho} \<cm>} 
			+ \CP^{\operatorname{e}_{(\sh,k)}}\del[1]{G(W) [\<wn> + \rho \<cm2>]} 
		\end{equation*}
		which then implies
		\begin{equation}
			U^{(1,1)} = \CP^{\operatorname{e}_{(\sh,k)}}\del[1]{G'(W) U^{(1,1)} \<cm> + G(W) \<cm2>}.
			\label{eq:U11_abstr}
		\end{equation}
		Also, note that one can immediately infer that~$U^{(1,2)} = 0$ because it solves a linear,~\emph{homogeneous} FP eq.
		\item Let~$i=2$. Similarly as before, we have
		\begin{equation*}
			U^{(2)}_\rho = 
			\CP^{\operatorname{e}_{(\sh,k)}}\del[1]{G'(W) U^{(2)}_\rho \<cm>
				+ G''(W) \sbr[1]{U^{(1)}_\rho}^2 \<cm>
				+ 2 G'(W) U^{(1)}_\rho [\<wn> + \rho \<cm2>]} 
		\end{equation*}
		and then
		\begin{align*}
			\partial_\rho U_\rho^{(2)}
			& =
			\CP^{\operatorname{e}_{(\sh,k)}}
			\del[1]{
				G'(W) \del[1]{\partial_\rho U^{(2)}_\rho} \<cm> 			
				+ 2 G''(W) \sbr[1]{\partial_\rho U^{(1)}_\rho} U^{(1)}_\rho \<cm>
			} \\
			& +
			\CP^{\operatorname{e}_{(\sh,k)}}
			\del[1]{
				2 G'(W) (\partial_\rho U^{(1)}_\rho) [\<wn> + \rho \<cm2>]
				+ 2 G'(W) U^{(1)}_\rho \<cm2> 
			}
		\end{align*}
		Finally, using that~$U^{(1,2)} = 0$, we obtain the formula
		\begin{align*}
			U^{(2,2)}
			& =
			\CP^{\operatorname{e}_{(\sh,k)}}
			\del[1]{
				G'(W) U^{(2,2)} \<cm> 
				+ 4 G'(W) U^{(1,1)} \<cm2>
				+ 2 G''(W) \sbr[1]{U^{(1,1)}}^2 \<cm>
			}. 
		\end{align*}
	\end{itemize}
	Finally we observe that~$U^{(i,i)}$ \emph{formally} agrees with~$i U^{(i)}$, cf.~\cite[Coro.~2.28]{friz_klose_22} again, except that in the defining fixed-point equation we have replaced~$\<wn> \rightsquigarrow \<cm2>$. 
	While~$U^{(i,i)}$ and~$iU^{(i)}$ live in modelled distribution spaces w.r.t. $E_{(\sh,k)}\hbz(\omega)$ and~$E_\sh \LL(k)$, resp. -- that is: in different Banach spaces -- their reconstructions live in the~\emph{same} space again.
	In fact, the previous arguments show that they coincide; the claim follows.
\end{proof}

We are ready to prove Proposition~\ref{prop:Qh_second_chaos_aux2}.

\begin{proof}
	Recall from~\cite[Lem.~3.20]{cfg} that~$\hbz(\omega+h) = T_h \hbz(\omega)$ holds for all~$h \in \CH$ and~$\mu$-a.e. $\omega \in B$. By definition of~$Q_\sh$, we thus have
	\begin{equation*}
		Q_\sh\del[1]{\hbz(\omega + \rho k)}
		=
		DF\sVert[0]_{w_\sh}\sbr[1]{\CR^{E_\sh T_{\rho k} \hbz(\omega)} U^{(2)}}
		+
		D^2 F\sVert[0]_{w_\sh}\sbr[1]{\CR^{E_\sh T_{\rho k} \hbz(\omega)} U^{(1)}}^2
	\end{equation*}
	and, by (bi-)linearity of~$DF\sVert[0]_{w_\sh}$ and~$D^2 F\sVert[0]_{w_\sh}$ in conjunction with Lemmas~\ref{lem:aux0} -- \ref{lem:aux2},
	\begin{align*}
		\partial^2_\rho\sVert[0]_{\rho = 0} Q_\sh\del[1]{\hbz(\omega + \rho k)}
		& =
		DF\sVert[0]_{w_\sh}\sbr[1]{2 \CR^{E_\sh \LL(k)} U^{(2)}}
		+
		2 D^2 F\sVert[0]_{w_\sh}\sbr[1]{\CR^{E_\sh \LL(k)} U^{(1)}}^2 \\
		& =
		2 Q_\sh(\LL(k))
		=
		2 A_\sh(k,k).
	\end{align*}
	Note that~$A_\sh(k,k)$ is~\emph{deterministic}. In particular, this implies that one can exchange the order of differentiation and expectation in~\eqref{prop:Qh_second_chaos_aux2:eq1}; essentially, this follows by dominated convergence, cf.~\cite[Coro.~$2.8.7$ and~Ex.~$2.12.68$]{bogachev}.  
	The proof is complete.
\end{proof}

\begin{proposition} \label{prop:Qh_second_chaos_aux1}
	For all~$k,\ell \in \CH$, we have
	\begin{equation}
		\partial_{\nu}\sVert[0]_{\nu = 0} \partial_{\rho}\sVert[0]_{\rho = 0} \E\sbr[1]{Q_\sh\del[1]{\hbz(\cdot + \rho k + \nu \ell)}}
		=
		2 C(k,\ell)
		\label{prop:Qh_second_chaos_aux1:eq1}.
	\end{equation}
\end{proposition}

\begin{proof} \label{lem:contrib_QH_zero_chaos}
	For~$\mu = \operatorname{Law}(\xi)$ as before,~$\nu, \rho \in \R$, and~$k, \ell \in \CH$, we let
	\begin{equation*}
		\mu_{\rho k + \nu \ell} := (\mathfrak{T}_{\rho k + \nu \ell})_* \mu
	\end{equation*}
	be the push-forward of~$\mu$ under the translation
	\begin{equation*}
		\mathfrak{T}_{\rho k + \nu \ell}: B \to B, \quad \mathfrak{T}_{\rho k + \nu \ell} \, \omega := \omega + \rho k + \nu \ell.
	\end{equation*}
	By~\cite[Lem.~$3.20$]{cfg} we have~$\hbz(\omega + \rho k + \nu \ell) = T_{\rho k + \nu \ell} \hbz(\omega)$ for~$\mu$-a.e.~$\omega \in B$. With~$\hat{Q}_\sh = Q_\sh(\hbz)$ in mind, we therefore have
	\begin{equation*}
		\E_\mu\sbr[1]{Q_\sh(\hbz)(\cdot + \rho k + \nu \ell)} 
		= \E_\mu\sbr[1]{Q_\sh(\hbz(\cdot + \rho k + \nu \ell))} 
		= E_{\mu_{\rho k + \nu \ell}}\sbr[1]{\hat{Q}_\sh}.
	\end{equation*}
	By the Cameron--Martin Theorem~\cite[Prop.~2.26]{daprato-zabczyk}, we know that
	\begin{equation*}
		\frac{\dif \mu_{\rho k + \nu \ell}}{\dif \mu}(\omega) = \exp\del[1]{-\II(\rho k + \nu \ell) + \rho \omega(k) + \nu \omega(\ell)}
	\end{equation*}
	where~$\omega(k) = I(k)(\omega)$ with the Paley-Wiener map~$I$ and~$\II = \nicefrac{1}{2} \norm[0]{\cdot}_\CH^2$ as before. Hence,
	\begin{equation*}
		\E_{\mu_{\rho k + \nu \ell}} \sbr[1]{\hat{Q}_\sh}
		=
		\int_B \hat{Q}_\sh(\omega) \mu_{\rho k + \nu \ell}(\dif \omega)
		=
		\int_B \hat{Q}_\sh(\omega) \exp\del[1]{-\II(\rho k + \nu \ell) + \rho \omega(k) + \nu\omega(\ell)} \mu(\dif \omega).
	\end{equation*}
	Recall that we found in the proof of proposition~\ref{prop:Qh_second_chaos_aux2} that one can exchange the order of differentiation and expectation in~\eqref{prop:Qh_second_chaos_aux1:eq1} and note that
	\begin{equation*}
		\II(\rho k + \nu \ell) = \frac{1}{2}\del[1]{\rho^2 \norm[0]{k}_\CH^2 + 2 \nu \rho \scal{k,\ell}_{\CH} + \nu^2 \norm[0]{\ell}_\CH^2}.
	\end{equation*}
	Therefore, we can easily compute the following derivative:
	\begin{equation}
		\partial_{\nu}\sVert[0]_{\nu = 0} \partial_{\rho}\sVert[0]_{\rho = 0} \exp\del[1]{-\II(\rho k + \nu \ell) + \rho \omega(k) + \nu\omega(\ell)} 
		=
		\omega(k) \omega(\ell) - \scal{k,\ell}_\CH
		\label{lem:contrib_QH_zero_chaos:eq1}
	\end{equation}
	Recall that~$\hat{Q}_\sh = \E_\mu\sbr[0]{\hat{Q}_\sh} + I_2(C)$ which implies the identity
	\begin{equation}
		\int_B \hat{Q}_\sh(\omega) \omega(k)\omega(\ell) \mu(\dif \omega)
		=
		\E_\mu\sbr[0]{\hat{Q}_\sh} \scal{k,\ell}_\CH
		+
		\int_B I_2(C)(\omega) \omega(k)\omega(\ell) \mu(\dif \omega)
		\label{lem:contrib_QH_zero_chaos:eq2}
	\end{equation}
	where the first summand is due to the fact that~$\E_\mu[\xi(k)\xi(\ell)] = \scal{k,\ell}_\CH$. 
	It is a consequence of standard Wiener calculus~(see for example~\cite[Prop.~$1.1.2$]{nualart}) that 
	\begin{equation*}
		\omega(k)\omega(\ell) = I_2(k \otimes \ell)(\omega) + \scal{k,\ell}_\CH,
	\end{equation*}
	which then gives 
	\begin{align}
		\int_B I_2(C)(\omega) \omega(k)\omega(\ell) \mu(\dif \omega)
		& 
		=
		\int_B I_2(C)(\omega) I_2(k \otimes \ell) \mu (\dif \omega) + \scal{k,\ell}_\CH \E_\mu[I_2(C)] \label{lem:contrib_QH_zero_chaos:eq3} \\
		&
		=
		\E_\mu\sbr[1]{I_2(C) I_2(k \otimes \ell)}
		\notag
	\end{align}
	because~$\E_\mu[I_2(C)] = 0$. Recall from~\cite[p.$9$, pt.~(iii)]{nualart} that
	\begin{equation}
		\E_\mu\sbr[1]{I_2(e_n \otimes e_n) I_2(k \otimes \ell)} = 2\scal{e_n,k}_\CH \scal{e_n,\ell}_\CH
		\label{lem:contrib_QH_zero_chaos:eq4}
	\end{equation}
	so that~$\E_\mu\sbr[1]{I_2(C) I_2(k \otimes \ell)} = 2\scal{C, k \otimes \ell}_{\CH^{\otimes 2}} = 2C(k,\ell)$. Combining this observation with the identities in~\eqref{lem:contrib_QH_zero_chaos:eq1} -- \eqref{lem:contrib_QH_zero_chaos:eq4}, the claim follows.
\end{proof}

\subsection{The component in the zero-th chaos} \label{sec:comp_zeroth_chaos}

In order to have a full understanding of the chaos decomposition for~$\hat{Q}_\sh$, it remains to calculate the zero-th chaos component or, in other words, the expected value~$\E\sbr[0]{\hat{Q}_\sh}$.

As we remarked in the introduction, and rigorously proved in Theorem~\ref{prop:Qh_second_chaos}, the operator~$A_\sh$ given in~\eqref{eq:op_A} characterises~$\pi_2 (\hat{Q}_\sh)$. 
It is well-known that iterated stochastic integrals obey an \emph{It\^{o}-type} chain rule, see for example~\cite[App.~A]{chandra_weber} for a concise introduction, resulting in a lower-order correction~$\pi_0(A_\sh[\xi_\d,\xi_\d]) \neq 0$ at the level of the mollified noise~$\xi_\d$.
Clearly, given that the former expression is \emph{quadratic} in the noise, one cannot in general expect convergence as~$\d \to 0$ and, instead, needs to~\emph{renormalise} certain ill-defined products. However, close inspection of~$A_\sh[\xi_\d,\xi_\d]$ suggests that this is not the case for all the products that appear in its definition. Let us give some heuristics to assist the reader in gaining an intuition.

\paragraph*{Heuristics.} The following considerations are \emph{not rigorous}: \label{para:heuristics}
\begin{enumerate}[label=(\arabic*)]
	\item \label{heur:1} Certainly,~$v_{\sh,\xi_\d} = u_\sh^{(1)}(\bz^{\xi_\d})$ from~\eqref{eq:vhk} needs no renormalisation, for it is \emph{linear} in the noise, see~\cite[Lem.~$2.34$]{friz_klose_22} for a proof.
	\item \label{heur:2} Increased regularity of~$\sh$ established in Corollary~\ref{lemma:Nzero} suggests that it is plausible for all the products in~\eqref{eq:vhkl:1} to be well-defined, even when~$\d \to 0$. (They are, see proposition~\ref{prop:q_trace_class} below.)
	\item \label{heur:3} On the other hand, the products in~\eqref{eq:vhkl:2} are ill-defined. We re-write eq.~\eqref{eq:vhk} as
	\begin{equation*}
		v_{\sh,\xi_\d}(t) 
		= \sbr[1]{\operatorname{Id} - R}^{-1}\del[3]{\int_0^\cdot P_{\cdot - s}(g(w_\sh(s))\xi_\d) \dif s}
		= \sum_{k=0}^\infty R^{\circ k} \del[3]{\int_0^\cdot P_{\cdot - s}(g(w_\sh(s))\xi_\d) \dif s}
	\end{equation*}
	by \emph{Neumann series} for some appropriate operator~$R$, ignoring issues of convergence. Taking only the constant term ($k = 0$) into account, we have
	\begin{equation*}
		v_{\sh,\xi_\d}(s) \approx \int_0^s P_{s-r}\del[1]{g(w_\sh(r)) \xi_\d} \dif r
	\end{equation*}	
	at~\enquote{zero-th order}. Hence, the product~$v_{\sh,\xi_\d}(s)\xi_\d$ in~\eqref{eq:vhkl:2} contains a summand of type\footnote{Here, $K$ is the Green's function of the Laplacian~$\Delta$ on~$\T^2$.} \mbox{$(K * \xi_\d)\xi_\d$} which needs to be renormalised as~$\d \to 0$.
	This is consistent with the fact that the gPAM regularity structure encodes such a product by the symbol~$\<11>$, the only one that needs re\-nor\-ma\-li\-sa\-tion.
\end{enumerate}
These observations suggest to \enquote{outsource} the singular products from~\eqref{eq:vhkl:2} into another operator~$\tilde{A}_\sh$, so that $A_\sh - \tilde{A}_\sh$ only contains products that need not be renormalised.
Returning to the side of rigour, we make these heuristics precise.

\paragraph*{--- Splitting off the \enquote{singular part}~$\tilde{A}$ from~$A$.}	

As before, let~$k,\ell \in \CH$. In accordance with our heuristics, we decompose~$v_{\sh,k,\ell}$ from~\eqref{eq:vhkl:1}--\eqref{eq:vhkl:2} into $v_{\sh,k,\ell} = v_{\sh,k,\ell}^{(\<cms>)} + v_{\sh,k,\ell}^{(\<wns>)}$ where
\begin{align}
	v_{\sh,k,\ell}^{(\<cms>)}(t)
	& =
	\int_0^t P_{t-s}\sbr[1]{g'(w_\sh(s)) v_{\sh,k,\ell}^{(\<cms>)}(s) \sh} \dif s 
	+ 
	\int_0^t P_{t-s}\sbr[1]{g''(w_\sh(s)) v_{\sh,k}(s) v_{\sh,\ell}(s) \sh} \dif s \label{e:vhkl_cm}\\
	v_{\sh,k,\ell}^{(\<wns>)}(t)
	& =
	\int_0^t P_{t-s}\sbr[1]{g'(w_\sh(s)) v_{\sh,k,\ell}^{(\<wns>)}(s) \sh} \dif s 
	+ 
	\int_0^t P_{t-s}\sbr[1]{g'(w_\sh(s)) v_{\sh,k}(s) \ell
		+ g'(w_\sh(s)) v_{\sh,\ell}(s) k} \dif s \notag
\end{align}
with~$v_{\sh,k,\ell}^{(\<wns>)}$ the \emph{singular} and $v_{\sh,k,\ell}^{(\<cms>)}$ the \emph{non-singular} part. (The reason for this choice of super-scripts will become clear in~eq.'s~\eqref{eq:decomp_U2} -- \eqref{eq:decomp_U2:reconstr} below.)
We then define the operator
\begin{equation}
	\tilde{A}_\sh[k,\ell] := DF\sVert[0]_{w_\sh}\del[1]{v_{\sh,k,\ell}^{(\<wns>)}}.
	\label{eq:Atilde}
\end{equation}
Recall that $B = \CC^{-1-\kappa}(\T^2)$ is the Banach space in the abstract Wiener space~$(B,\CH,\mu)$.	

\begin{proposition} \label{prop:q_trace_class}
	The bilinear form~$q_\sh := A_\sh - \tilde{A}_\sh$ given by 
	\begin{equation}
		q_\sh[k,\ell] 
		= DF\sVert[0]_{w_\sh}\sbr[1]{v_{\sh,k,\ell}^{(\<cms>)}}
		+ 
		D^2 F\sVert[0]_{w_\sh}\sbr[0]{v_{\sh,k}, v_{\sh,\ell}}
		\label{eq:def_q}
	\end{equation}
	is trace-class and can be continuously extended to act on~$B \x B$, with the same operator norm.
\end{proposition}	

\begin{remark}
	One should think of~$k,\ell$ as typical realisations~$\xi(\omega) \in B$ of the spatial white noise~$\xi$.
\end{remark}

For the proof that follows, recall the results from Corollary~\ref{lemma:Nzero}.
\begin{proof}
	In order to apply Goodman's Theorem~\cite[Thm.~$4.6$]{kuo} and infer the trace-class property of~$q_\sh$, we need to establish the estimate
	\begin{equation} \label{eq:est_goodman}
		\abs[0]{q_\sh[k,\ell]} \aac \norm[0]{k}_{B} \norm[0]{\ell}_B.
	\end{equation}
	\paragraph*{-- Analysis of~$v_{\sh,k}$.}
	We first amend and extend the analysis of~$v_{\sh,k}$ presented in Subsection~\ref{sec:der_eq:gpam} and, similarly to the proof of Proposition~\ref{prop:summary_der_eq}, omit the standard fixed-point argument. 
	Instead, we focus again on the spatial regularity of~$v_{\sh,k}$ and now consider~$k \in B$ (rather than~$k \in H^{-\theta}$) which presents the same kind of problem as before but leads to different regularity results. 
	\begin{itemize}[itemsep=5pt]
		\item By looking at the analysis of the term~$J_{\sh,k}$ in~\eqref{eq:vhk_mild}, we see that for~$g(w_\sh(s))k \in \CC^{-1-\kappa}$ to be well-defined by Proposition~\ref{prop:mult_besov}, we require
		\begin{equation} \label{eq:reg_min_product}
			(\gamma + 1 - \kappa) + (-1-\kappa) > 0 
		\end{equation}
		if we assume that~$\sh \in H^\gamma$ and~$w_\sh \in C_T\CC^{\gamma + 1 - \kappa}$, cf.~Proposition~\ref{prop:ex_det_gpam_Hgamma}, and combine that with Lemma~\ref{lem:comp_sobolev}. 
		The previous condition is equivalent to~$\gamma > 2 \kappa$; therefore, we apply Corollary~\ref{lemma:Nzero} (which is valid under the conditions of Theorem~\ref{thm:main_result}) to get~$\sh \in H^{\gamma_\star}$ for~$\gamma_\star := 3\kappa$, see also Remark~\ref{rmk:reg_u_0} below.
		\begin{center}
			$\boxed{\text{This is where we crucially require better than Cameron--Martin regularity of~$\sh \in H^{\gamma_\star}$.}}$
		\end{center}
		As in~\eqref{eq:aux_1}, the heat kernel estimates of Proposition~\ref{prop:heat_besov} then imply the estimate
		\begin{equs}[][eq:Jhk_Goodman]
			\thinspace &
			\norm[0]{J_{\sh,k}(t)}_{\CC^{1-2\kappa}} 
			= \,
			\norm[3]{\int_0^t P_{t-s} [g(w_\sh(s))k] \dif s}_{\CC^{1-2\kappa}} 
			\aac
			\int_0^t \norm[0]{P_{t-s} [g(w_\sh(s))k]}_{\CC^{1-2\kappa}} \dif s \notag \\
			\aac \thinspace &
			\int_0^t (t-s)^{\frac{1}{2}(-2 + \kappa)} \norm[0]{g(w_\sh(s))k}_{\CC^{-1 - \kappa}} \dif s 
			\aac
			\frac{2}{\kappa} T^{\frac{\kappa}{2}}
			\norm[0]{g(w_\sh)}_{C_T \CC^{1 +2\kappa}} \norm[0]{k}_{B}
		\end{equs}
		\emph{At best}, we therefore have~$v_{\sh,k} \in C_T \CC^{1-2\kappa}$ which is our standing assumption for now.
		\item We now turn to~$I_{\sh,k}$ in~\eqref{eq:vhk_mild} and, since $g'(w_\sh) \in C_T \CC^{1+2\kappa}$ and~$\sh \in H^{3\kappa}$, we find
		\begin{equation*}
			g'(w_\sh(s))v_{\sh,k}(s) \in \CC^{1-2\kappa}, \quad 
			g'(w_\sh(s))v_{\sh,k}(s)\sh \in H^{3\kappa} \embed \CC^{-1+2\kappa}.
		\end{equation*}
		Like in~\eqref{eq:reg_min_term_B_case1}, we then find
		\begin{equs}[eq:Ihk_Goodman]
			\norm[0]{I_{\sh,k}(t)}_{\CC^{1 -2\kappa}}
			& = \,
			\norm[3]{\int_0^t P_{t-s} [g'(w_\sh(s)) v_{\sh,k}(s)\sh] \dif s}_{\CC^{1 -2\kappa}} \\
			& \lesssim \,
			\int_0^t (t-s)^{2\kappa} \norm[0]{g'(w_\sh(s)) v_{\sh,k}(s)\sh}_{\CC^{-1 + 2\kappa}} \dif s \notag\\
			& \aac \, 
			\norm{g'(w_\sh)}_{C_T \CC^{1 + 2\kappa}} \norm[0]{\sh}_{H^{3\kappa}}
			\int_0^t (t-s)^{2\kappa} \norm[0]{v_{\sh,k}(s)}_{\CC^{1 -2\kappa}} \dif s.
		\end{equs}
		\item Finally, we combine~\eqref{eq:Jhk_Goodman} and~\eqref{eq:Ihk_Goodman} and apply Gronwall's Lemma.
		Altogether, like in~\eqref{prop:summary_der_eq:gronwall} we find that
		\begin{equation} \label{eq:vhk_Goodman}
			\norm[0]{v_{\sh,k}(t)}_{\CC^{1 -2\kappa}}
			\leq C  \exp\del[3]{C \int_0^t (t-s)^{2\kappa} \dif s} \norm[0]{k}_{B}
			= \tilde{C} \norm[0]{k}_{B}
			\aac
			\norm[0]{k}_{B}
		\end{equation}
		where~$C$ is as defined in~\eqref{prop:summary_der_eq:value_constant} with~$\gamma = \gamma_\star = 3\kappa$.
	\end{itemize}
	These observations allow to deal with the second summand in~\eqref{eq:def_q} since the same arguments apply to~$v_{\sh,\ell}$ for~$\ell \in B$. 
	\paragraph*{-- Analysis of~$v_{\sh,k,\ell}^{(\<cms>)}$.}
	As before, we omit the fixed-point argument and focus on the spatial regularities.
	For the first summand in~\eqref{eq:def_q}, we analyse the products within the equation for~$v_{\sh,k,\ell}^{(\<cms>)}$ in~\eqref{e:vhkl_cm} in the same way, starting with the \emph{inhomogeneous part}.
	\begin{itemize}[itemsep=5pt]
		\item We have~$v_{\sh,k}(s)v_{\sh,\ell}(s) \in \CC^{1-2\kappa}$ and then~$v_{\sh,k}(s)v_{\sh,\ell}(s)\sh \in H^{3\kappa}$ since~$\sh \in H^{3\kappa}$.
		\item Since~$g''(w_\sh) \in C_T\CC^{1 + 2\kappa}$ by Lemma~\ref{lem:comp_sobolev}, it follows that~$g''(w_\sh(s)) v_{\sh,k}(s)v_{\sh,\ell}(s)\sh \in H^{3\kappa} \embed \CC^{-1+2\kappa}$.
		\item As before, by the regularising effect of the heat semi-group we then have
		\begin{equs}[e:vhkl_cm_inhom]
			\thinspace &
			\norm[3]{\int_0^t P_{t-s}\sbr[1]{g''(w_\sh(s)) v_{\sh,k}(s) v_{\sh,\ell}(s) \sh} \dif s}_{\CC^{1+\kappa}} \\[5pt]
			\aac \thinspace & \thinspace
			\int_0^t \norm[0]{P_{t-s} [g''(w_\sh(s)) v_{\sh,k}(s) v_{\sh,\ell}(s) \sh]}_{\CC^{1+\kappa}} \dif s \\[5pt]
			\aac \thinspace & \thinspace
			\int_0^t (t-s)^{\frac{1}{2}(-2 + \kappa)} \norm[0]{g''(w_\sh(s)) v_{\sh,k}(s) v_{\sh,\ell}(s) \sh}_{\CC^{-1 + 2\kappa}} \dif s  \\[5pt]
			\aac \thinspace & \thinspace
			\frac{2}{\kappa} T^{\frac{\kappa}{2}}
			\norm[0]{g(w_h)}_{C_T \CC^{1 +2\kappa}} 
			\norm[0]{v_{\sh,k}}_{C_T \CC^{1 -2\kappa}}
			\norm[0]{v_{\sh,\ell}}_{C_T \CC^{1 -2\kappa}} 
			\norm[0]{\sh}_{H^{3\gamma}} \\[5pt]
			\aac \thinspace & \thinspace
			\tilde{C}^2
			\frac{2}{\kappa} T^{\frac{\kappa}{2}}
			\norm[0]{g(w_h)}_{C_T \CC^{1 +2\kappa}} 
			\norm[0]{k}_B
			\norm[0]{\ell}_B
			\norm[0]{\sh}_{H^{3\gamma}}
		\end{equs}
		where we have used the estimate~\eqref{eq:vhk_Goodman} for~$v_{\sh,k}$ and~$v_{\sh,\ell}$ in the last step.		
		\emph{At best}, we therefore have~$v_{\sh,k,\ell}^{(\<cms>)} \in C_T \CC^{1+\kappa}$, our standing assumption from now on.
	\end{itemize}
	Concerning the \emph{homogeneous part} of~$v_{\sh,k,\ell}^{(\<cms>)}$, we have:
	\begin{itemize}[itemsep=5pt]
		\item The products~$g'(w_\sh(s)) v_{\sh,k,\ell}^{(\<cms>)}(s) \in \CC^{1+\kappa}$ and~$g'(w_\sh(s)) v_{\sh,k,\ell}^{(\<cms>)}(s)\sh \in H^{3\kappa} \embed \CC^{-1+2\kappa}$ are well-defined.
		\item The regularising effect of the heat semi-group implies that
		\begin{equs}[e:vhkl_cm_hom]
			\thinspace &
			\norm[3]{\int_0^t P_{t-s}\sbr[1]{g'(w_\sh(s)) v_{\sh,k,\ell}^{(\<cms>)}(s) \sh} \dif s}_{\CC^{1+\kappa}} 
			\aac
			\int_0^t \norm[0]{P_{t-s} [g'(w_\sh(s)) v_{\sh,k,\ell}^{(\<cms>)}(s) \sh]}_{\CC^{1+\kappa}} \dif s \\[5pt]
			\aac \thinspace & \thinspace
			\int_0^t (t-s)^{\frac{1}{2}(-2 + \kappa)} \norm[0]{g'(w_\sh(s)) v_{\sh,k,\ell}^{(\<cms>)}(s) \sh}_{\CC^{-1 + 2\kappa}} \dif s  \\[5pt]
			\aac \thinspace & \thinspace
			\norm[0]{g'(w_{\sh})}_{C_T \CC^{1 +2\kappa}} 
			\norm[0]{\sh}_{H^{3\gamma}} 
			\int_0^t
			(t-s)^{\frac{1}{2}(-2 + \kappa)}
			\norm[0]{v_{\sh,k,\ell}^{(\<cms>)}(s)}_{\CC^{1 + \kappa}} \dif s 
		\end{equs}
	\end{itemize}
	We can now combine the fixed-point equation~\eqref{e:vhkl_cm} for~$v_{\sh,k,\ell}^{(\<cms>)}$ with the estimates~\eqref{e:vhkl_cm_inhom} and~\eqref{e:vhkl_cm_hom} to obtain
	\begin{equation} \label{e:vhkl_cm_combined}
		\norm[0]{v_{\sh,k,\ell}^{(\<cms>)}(t)}_{\CC^{1 + \kappa}} 
		\leq
		\bar{C}\del[2]{
			\norm[0]{k}_B
			\norm[0]{\ell}_B
			+
			\int_0^t
			(t-s)^{\frac{1}{2}(-2 + \kappa)}
			\norm[0]{v_{\sh,k,\ell}^{(\<cms>)}(s)}_{\CC^{1 + \kappa}} \dif s} 
	\end{equation}
	where
	\begin{equation*}
		\bar{C} 
		\simeq 
		\tilde{C}^2
		\frac{2}{\kappa} T^{\frac{\kappa}{2}}
		\norm[0]{g(w_h)}_{C_T \CC^{1 +2\kappa}}
		\norm[0]{\sh}_{H^{3\gamma}}
		+
		\norm[0]{g'(w_{\sh})}_{C_T \CC^{1 +2\kappa}} 
		\norm[0]{\sh}_{H^{3\gamma}}. 
	\end{equation*}
	Gronwall's Lemma applied to~\eqref{e:vhkl_cm_combined} then leads to the estimate
	\begin{equation} \label{e:vhkl_cm_gronwall}
		\norm[0]{v_{\sh,k,\ell}^{(\<cms>)}(t)}_{\CC^{1 + \kappa}} 
		\leq
		\bar{C}
		\norm[0]{k}_B
		\norm[0]{\ell}_B
		\exp\del[3]{\bar{C} \int_0^t (t-s)^{\frac{1}{2}(-2+\kappa)} \dif s}
		\aac 
		\norm[0]{k}_B
		\norm[0]{\ell}_B
	\end{equation}
	\paragraph*{-- The final bound.}
	In order to obtain the desired bound~\eqref{eq:est_goodman}, recall the assumptions on~$F$ in Theorem~\ref{thm:main_result}.
	In particular, they can be seen to imply (since~$1-5\kappa < \mu$ for~$\mu \in \{1 + \kappa, 1-2\kappa\}$) that
	\begin{equation*}
		DF\sVert[0]_{w_\sh} \in \CL\del[1]{C_T \CC^{1+\kappa},\R}, \quad
		D^2F\sVert[0]_{w_\sh} \in \CL^{(2)}\del[1]{C_T \CC^{1-2\kappa},\R}
	\end{equation*}
	which, combined with~\eqref{eq:vhk_Goodman} and~\eqref{e:vhkl_cm_gronwall}, lead to the estimate
	\begin{equs}
		\abs[0]{q_\sh[k,\ell]} 
		& \aac
		\norm[0]{D^2F\sVert[0]_{w_\sh}}_{\CL^{(2)}\del[0]{C_T \CC^{1-2\kappa},\R}}
		\norm[0]{v_{\sh,k}(t)}_{\CC^{1 -2\kappa}} \norm[0]{v_{\sh,\ell}(t)}_{\CC^{1 -2\kappa}} \\
		& \quad +
		\norm[0]{DF\sVert[0]_{w_\sh}}_{\CL\del[0]{C_T \CC^{1+\kappa},\R}} 
		\norm[0]{v_{\sh,k,\ell}^{(\<cms>)}(t)}_{\CC^{1 + \kappa}} \\[5pt]
		& \aac \norm[0]{k}_{B} \norm[0]{\ell}_B.
	\end{equs} 
	The claim in~\eqref{eq:est_goodman} follows and the addendum is a trivial consequence of that estimate by bilinearity of~$q_\sh$.
\end{proof}		

\begin{remark} \label{rmk:reg_u_0}
	In the proof of the previous proposition, various arguments -- notably~\eqref{eq:reg_min_product} -- require the regularity condition $g^{(m)}(w_\sh) \in C_T\CC^{1 + 2\kappa}$ for~$m \in \{0,1,2\}$, where the exponent is~$1+2\kappa = 1 - \kappa + \gamma_\star$ for~$\gamma_\star = 3 \kappa$.  
	In turn, this requires increased regularity of the minimiser, namely~$\sh \in H^{\gamma_\star}$, as well as~$u_0 \in \CC^{1+2\kappa}(\T^2)$ for Proposition~\ref{prop:ex_det_gpam_Hgamma} to apply; the validity of the previous statements under the assumptions of Theorem~\ref{thm:main_result} has been checked in Corollary~\ref{lemma:Nzero} above. 
	
	However, let us emphasise that, in the setting of~Theorem~\ref{thm:main_result}, we only require~$u_0 \in \CC^{1-\kappa}(\T^2)$ to conclude that $\sh \in H^{\gamma_\star}$ -- hence, there is no contradiction to Theorem~\ref{thm:reg_minimiser}.
\end{remark}

Let~$\tilde{q}_\sh \in \CL(\CH)$ be the operator associated to~$q_\sh$ via~$\scal{\tilde{q}_\sh h_1,h_2}_\CH = q_\sh[h_1,h_2]$ for all~$h_1,h_2 \in \CH$. 
Since~$\tilde{q}_\sh$ is trace-class, it is also compact, so the well-known Spectral Theorem applies:  We denote the eigenvalues of~$\tilde{q}_\sh$ by~$(\mu_k)_{k \in \N} \in \ell^1(\R)$ and the corresponding eigenbasis by~$(e_k)_{k \in \N} \subseteq \CH$.

We will write~$q_\sh(h) := q_\sh[h,h]$ and note that~$\CH^* \ni q_\sh \neq \tilde{q}_\sh \in \CL(\CH)$.

\begin{lemma}
	The identity
	\begin{equation*}
		q_\sh(\xi) = \sum_{n \in \N} \mu_n \del[0]{\xi(e_n)^2-1} + \operatorname{Tr} q_\sh 
	\end{equation*}
	holds almost-surely. In particular,~$\E[q_\sh(\xi)] = \operatorname{Tr} q_\sh$.
\end{lemma}

\begin{proof}
	For any~$h \in \CH$, it is clear that
	\begin{equation*}
		q_\sh(h) 
		= \scal{\tilde{q}_\sh h,h}_\CH
		= \left \langle \sum_{n \in \N} \mu_n \scal{h,e_n}_\CH e_n, h \right\rangle_\CH
		= \sum_{n \in \N} \mu_n \scal{h,e_n}_\CH^2
		= \sum_{n \in \N} \mu_n \del[1]{\scal{h,e_n}_\CH^2 - 1} + \operatorname{Tr} q_\sh.
	\end{equation*}
	The claim follows by extension of~$q_\sh$ to~$B$ as in proposition~\ref{prop:q_trace_class}. The addendum is true because
	\begin{equation*}
		\xi(e_n)^2 - 1 = I_2(e_n \tp e_n)	
	\end{equation*}
	has vanishing expectation.
\end{proof}

In a next step, we will want to relate~$q_\sh(\xi)$ to the expression for~$\hat{Q}_\sh$ in~\eqref{eq:Q_h_repeated}. As we will see, this requires a systematic analysis of the renormalisation procedure.

\paragraph*{--- Relating~$q_\sh(\xi)$ to~$\hat{Q}_\sh$.}

For any~$\bz \in \MM$, we decompose
$U^{(2)}(\bz) 
= 
U^{(2,\<wns>)}(\bz) 
+
U^{(2;\<cms>)}(\bz)$
where
\begin{equs}[][eq:decomp_U2]
	U^{(2;\<wns>)}
	& = 
	\CP^{\esh}\del[1]{G'(W) U^{(2;\<wns>)} \<cm>} + 2 \CP^{\esh}\del[1]{G'(W) U^{(1)} \<wn>}, \\
	U^{(2;\<cms>)}
	& =
	\CP^{\esh}\del[1]{G'(W) U^{(2;\<cms>)} \<cm>}
	+
	\CP^{\esh}\del[1]{G''(W) \sbr[1]{U^{(1)}}^{\star 2} \<cm>}. \notag
\end{equs}
Recall from remark~\ref{rmk:rel_der} that~$u_{\xi_\d,\sh}^{(2)} = v_{\sh,\xi_\d,\xi_\d}$. Accordingly, the decomposition in~\eqref{eq:decomp_U2} is chosen in such a way that, for $\d > 0$, we have
\begin{equation}
	v_{\sh,\xi_\d,\xi_\d}^{(\sigma)} = \CR^{E_\sh \bz^{\xi_\d}} U^{(2;\sigma)}, \quad \sigma \in \{\<wns>, \<cms>\}.
	\label{eq:decomp_U2:reconstr}
\end{equation}
The following lemma makes point~\ref{heur:2} in the above heuristics precise: It states that all the products on the RHS of the equation for~$U^{(2;\<cms>)}$ are well-defined \emph{without} renormalisation.	

\begin{lemma} \label{lem:carleman_fredholm_non_renormalisation}
	The equality
	$\CR^{E_\sh \hbz^{\xi_\d}} U^{(2;\<cms>)} = \CR^{E_\sh \bz^{\xi_\d}} U^{(2;\<cms>)}$
	holds for any~$\d > 0$.
\end{lemma}

\begin{proof}
	Let~
	$Y := G'(W) U^{(2;\<cms>)} \<cm>
	+
	G''(W) \sbr[1]{U^{(1)}}^{\star 2} \<cm>$. For each~$z \in \R^3$, we have
	\begin{equation}
		\CR^{E_\sh \hbz^{\xi_\d}} U^{(2;\<cms>)} (z)
		=
		\CR^{E_\sh \hbz^{\xi_\d}} \CP Y (z)
		=
		(P * \CR^{E_\sh \hbz^{\xi_\d}} Y) (z)
		=
		\int_{\R^3} P(z-\bar{z}) (\CR^{E_\sh \hbz^{\xi_\d}} Y) (\bar{z}) \dif \bar{z}
		\label{aux_pf_1}
	\end{equation}
	as well as
	\begin{equation*}
		(\CR^{E_\sh \hbz^{\xi_\d}} Y) (\bar{z})
		=
		\del[1]{\hat{\Pi}^{\xi_\d}_{\bar{z}} Y(\bar{z})}(\bar{z})
		=
		\del[1]{\Pi^{\xi_\d}_{\bar{z}} M_\d Y(\bar{z})}(\bar{z})
		\overset{(*)}{=}
		\del[1]{\Pi^{\xi_\d}_{\bar{z}} Y(\bar{z})}(\bar{z})
		=
		(\CR^{E_\sh \bz^{\xi_\d}} Y) (\bar{z}).
		\label{aux_pf_2}
	\end{equation*}
	The starred identity is true because $\scal{Y,\<11>} = 0$ (due to the factor~$\<cm>$) while~$M_\d$ only acts non-trivially on the symbol~$\<11>$.
	Therefore, all the equalities in~\eqref{aux_pf_1} are also valid with~$\hbz^{\xi_\d}$ replaced by~$\bz^{\xi_\d}$ and the respective quantities are equal. The claim follows.
\end{proof}

As a consequence of the previous lemma, we obtain the following decomposition of~$\hat{Q}_\sh$.

\begin{proposition} \label{prop:chaos_decomp_hessian}
	The identity
	\begin{equation}
		\hat{Q}_\sh = q_\sh(\xi) + DF\sVert[0]_{w_\sh}\del[1]{\CR^{E_\sh \hbz} U^{(2;\<wns>)}} 
		\label{prop:chaos_decomp_hessian:eq1}
	\end{equation}
	holds almost surely. In particular, we have~$\E\sbr[0]{\hat{Q}_\sh} = \operatorname{Tr} q_\sh + \lambda$ with
	\begin{equation}
		\lambda := \E\sbr[1]{DF\sVert[0]_{w_\sh}\del[1]{\CR^{E_\sh \hbz} U^{(2;\<wns>)}}}.
		\label{prop:chaos_decomp_hessian:eq2}
	\end{equation}
\end{proposition}

\begin{proof}
	Let~$\d > 0$. By definition of~$\hat{Q}_{\sh;\d}$ and~$q_\sh$ (see proposition~\ref{prop:q_trace_class}) combined with Lemmas~\ref{lem:carleman_fredholm_non_renormalisation} and~\cite[Lem.~$2.34$]{friz_klose_22} (which says that~$\hat{u}_{\sh;\d}^{(1)} = u_{\sh;\d}^{(1)}$), we find
	\begin{align*}
		\thinspace
		&
		\hat{Q}_{\sh;\d} - DF\sVert[0]_{w_\sh}\del[1]{\CR^{E_\sh \hbz^{\xi_\d}} U^{(2;\<wns>)}} 
		=
		DF\sVert[0]_{w_\sh} (\CR^{E_\sh \hbz^{\xi_\d}} U^{(2;\<cms>)}) 
		+
		D^2 F\sVert[0]_{w_\sh} \sbr[1]{\hat{u}_{\sh;\d}^{(1)}, \hat{u}_{\sh;\d}^{(1)}} \\
		= \ & 
		DF\sVert[0]_{w_\sh} (\CR^{E_\sh \bz^{\xi_\d}} U^{(2;\<cms>)}) 
		+
		D^2 F\sVert[0]_{w_\sh} \sbr[1]{u_{\sh;\d}^{(1)}, u_{\sh;\d}^{(1)}}
		=
		q_\sh(\xi_\d).
	\end{align*}
	Note that we have also used remark~\ref{rmk:rel_der} and~eq.~\eqref{eq:decomp_U2:reconstr}. 
	The claim follows by sending~$\d \to 0$.
\end{proof}

Finally, we can prove the main result of this article, Theorem~\ref{thm:main_result}.

\begin{proof} \label{pf:thm_1}
	A generic expression for~$a_0$ has been given in corollary~\ref{coro:general}. The different terms appearing in it have been characterised in Theorem~\ref{prop:Qh_second_chaos} and proposition~\ref{prop:chaos_decomp_hessian}, respectively. The claimed formula~\eqref{thm:eq_a0} for~$a_0$ follows.   
\end{proof}

\subsection[Approximating~$\lambda$]{Approximating~$\lambda$}\label{subsec:eta}

Our objective was to obtain an \emph{explicit} formula for~$a_0$ --
yet, the definition of~$\lambda$ in~\eqref{prop:chaos_decomp_hessian:eq2} as the expectation of some operator applied to a reconstructed modelled distribution is quite intangible.
Therefore, in subsection~\ref{subsec:eta}, we will prove that~$\eta$ can be approximated by certain regularised versions~$\lambda_\d$ which are defined from classical objects and can thus be computed in practice.
We set
\begin{equation*}
	\tilde{u}^{(2)}_{\sh} := \CR^{E_\sh \hbz} U^{(2;\<wns>)},
	\quad
	\tilde{u}^{(2)}_{\sh;\d} := \CR^{E_\sh \hbz^{\xi_\d}} U^{(2;\<wns>)}.
\end{equation*} 
Recall from~\eqref{eq:decomp_U2} that~$U^{(2;\<wns>)}(\bz)$ satisfies a~\emph{linear} fixed-point equation in the space~$\DD^{\gamma,\eta}(E_\sh\bz)$ and, via its inhomogeneity, depends on~$U^{(1)}(\bz)$.
As we have sketched in the heuristics on page~\pageref{para:heuristics} above (and will further detail in the proof of proposition~\ref{prop:conv_eta} below), this necessitates to renormalise certain products: Despite the linearity of the equation, it is therefore impossible to write down a stochastic PDE for~$\tilde{u}_\sh^{(2)}$ (i.e. when~$\bz = \hbz$) and, a fortiori, to obtain a closed-form expression for~$\lambda$.

Instead, the theory of regularity structures allows to define~$\tilde{u}_\sh^{(2)}$ as the limit \emph{in probability} in~$\CX_T$ of~$\tilde{u}^{(2)}_{\sh;\d}$ as~$\d \to 0$. 
However, it is not clear whether that convergence also holds in~$L^1(\Omega;\CX_T)$ or, more generally, in~$L^p(\Omega;\CX_T)$ for~$p \geq 1$. 
That question is important for computing~$\lambda$ because 
\begin{itemize}[itemsep=3pt]
	\item one \emph{can} derive a linear stochastic PDE for~$\tilde{u}^{(2)}_{\sh;\d}$ when~$\d > 0$,
	\item solve it via standard techniques, and
	\item approximate~$\lambda_\delta := \E\sbr[0]{DF\sVert[0]_{w_\sh}\del[1]{\tilde{u}^{(2)}_{\sh;\d}}}$ numerically for~$\d \ll 1$, for example via Monte--Carlo methods.
\end{itemize}
The last step, in particular, would only be sensible if~$\lambda_\d \to \lambda$ as~$\d \to 0$. Indeed, that is true; the following proposition formalises the procedure we just outlined.

\begin{proposition}\label{prop:conv_eta}
	One can obtain~$\tilde{u}_{\sh;\d}^{(2)}$ by solving the following system of linear stochastic PDEs:
	\begin{equs}[][eq:pde_aux]
		(\partial_t - \Delta) \tilde{u}^{(2)}_{\sh;\d} 
		& = 
		2 g'(w_{\sh})u^{(1)}_{\sh;\d}  \xi_\d 
		- 2g'(w_{\sh})g(w_{\sh}) \fc_{\d} 
		+ \tilde{u}^{(2)}_{\sh;\d} g'(w_{\sh}) \sh \\
		(\partial_t - \Delta) u^{(1)}_{\sh;\d}
		& =
		g(w_{\sh})  \xi_\d 
		+ u^{(1)}_{\sh;\d} g'(w_{\sh}) \sh
	\end{equs}
	In addition, for any~$p \geq 1$ we have~$\tilde{u}^{(2)}_{\sh;\d} \to  \tilde{u}^{(2)}_{\sh}$ in~$L^p(\Omega;\CX_T)$ and then~$\lambda_\d \to \lambda$ as~$\d \to 0$. 
\end{proposition}

\begin{proof}
	The \emph{system} in~\eqref{eq:pde_aux} can be derived in complete analogy to~\eqref{eq:u_hd1} and~\eqref{eq:u_hd2} by applying the me\-thods of proof for \cite[Prop.~$2.47$]{friz_klose_22}. 
	Thus, we shall only outline the parts pertaining to the rigorous implementation of the renormalisation procedure.
	First, recall that~$U^{(1)}$ satisfies an equation like~$U^{(1,1)}$ in~\eqref{eq:U11_abstr}, with~$\<cm2> \rightsquigarrow \<wn>$, which can be written as
	\begin{equation*}
		U^{(1)} = \CI\del[1]{H(U^{(1)})} + \bar{\CT}, \quad H(U^{(1)}) := G'(W)U^{(1)} \<cm> + G(W) \<wn>
	\end{equation*}
	where~\enquote{$\bar{\CT}$} denotes polynomial components. Given that~$G^{(\ell)}(W) \in \scal{\1,\<1g>,X}$ for any~$\ell \in \N_0$ for which it is well-defined, it can easily be seen that
	\begin{equation*}
		\scal{H(U^{(1)}),\<11>} = 0, \quad  
		\scal{U^{(1)},\<1>} = \scal{H(U^{(1)}),\<wn>} = \scal{G(W)\<wn>,\<wn>} = \scal{G(W),\1} = g(w_\sh).
	\end{equation*} 
	In the same way, we then rewrite eq.~\eqref{eq:decomp_U2} for~$U^{(2;\<wns>)}$ as
	\begin{equation}
		U^{(2;\<wns>)} = \CI\del[1]{\bar{H}(U^{(2;\<wns>)})} + \bar{\CT}, 
		\quad 
		\bar{H}(U^{(2;\<wns>)}) := G'(W)U^{(2;\<wns>)} \<cm> + 2G'(W) U^{(1)} \<wn>
		\label{prop:conv_eta:pf_formula_U2wn}
	\end{equation}
	from which one may infer that
	\begin{equation*}
		\scal{\bar{H}(U^{(2;\<wns>)}), \<11>}
		=
		2 \scal{G'(W) U^{(1)} \<wn>, \<11>}
		=
		2 \scal{G'(W),\1} \scal{U^{(1)}, \<1>}
		= 
		2 g'(w_\sh)g(w_\sh).
	\end{equation*}
	As a consequence, for~$z = (t,x) \in [0,T] \x \T^2$ we have 
	\begin{equs}
		(\partial_t - \Delta) \tilde{u}_{\sh;\d}^{(2)}(z)
		& = 
		(\partial_t - \Delta) \CR^{E_\sh\hbz^{\xi_\d}}U^{(2;\<wns>)}(z)
		=
		\CR^{E_\sh\hbz^{\xi_\d}}(\bar{H}(U^{(2;\<wns>)}))(z) \\
		& =
		\hat{\Pi}_z^{\xi_\d;\esh} (\bar{H}(U^{(2;\<wns>)})(z))(z)
		=
		\Pi_z^{\xi_\d;\esh} M_\d (\bar{H}(U^{(2;\<wns>)})(z))(z) \\
		& =
		\Pi_z^{\xi_\d;\esh} (\bar{H}(U^{(2;\<wns>)})(z))(z) - \scal{\bar{H}(U^{(2;\<wns>)})(z), \<11>} \fc_\d 
		= (\ldots) - 2g'(w_\sh(z))g(w_\sh(z))\fc_\d 
	\end{equs}
	and thereby rigorously derived the renormalisation counter-term that appears in~\eqref{eq:pde_aux}.
	
	\emph{Convergence statements.}
	We know that~$\tilde{u}_{\sh;\d}^{(2)} \to \tilde{u}_{\sh}^{(2)}$ in~$\CX_T$ in probability, so we need to establish uniform integrability.
	To this end, we will prove uniform~$L^p$ boundedness for~$p \geq 1$, i.e. the existence of a constant~$M > 0$ such that
	\begin{equation}
		\sup_{\d \in (0,1)} \E\sbr[1]{\ \norm[1]{\tilde{u}_{\sh;\d}^{(2)}}_{\CX_T}^p} \leq M.
		\label{pf:prop:conv_eta:unif_Lp_bdd}
		\tag{$\star$}
	\end{equation}
	The proofs of~\cite[Prop.~$2.39$ and Thm.~$2$]{friz_klose_22} (for~$m=2$) without any additional arguments imply that\footnote{For technical reasons, we needed to work with the~\emph{minimal} BPHZ model~$\hbz^{\xi_\d,-} \in \MM_-$ and the \emph{homogeneous} model norm~$\barnorm{\cdot}$ on~$\MM_-$ in the quoted article. In the present case, however, we can work with the~\enquote{standard} notions from Hairer's original article~\cite{hairer_rs} in light of the estimate~$\barnorm{\hbz^{\xi_\d,-}} \leq 1 + \threebars \hbz^{\xi_\d,-} \threebars \leq 1 + \threebars \hbz^{\xi_\d} \threebars$. The latter is true because of the estimate in~\cite[eq.~(A.$18$)]{friz_klose_22} together with the elementary inequality~$\sqrt{x} \leq 1 + x$ for~$x \geq 0$.}
	\begin{equation*}
		\norm[1]{\tilde{u}_{\sh;\d}^{(2)}}_{\CX_T}^p 
		\aac
		\del[1]{1 + \barnorm{\hbz^{\xi_\d,-}}}^{2p}
		\aac
		\del[1]{2 + \threebars \hbz^{\xi_\d} \threebars}^{2p}.
	\end{equation*}
	The triangle inequality for the semi-norm~$\threebars \cdot \threebars$ in conjunction with the elementary estimate $(a+b)^p \leq 2^{p-1} (a^p + b^p)$ for~$a,b \geq 0$ then implies
	\begin{equation*}
		\del[1]{2 + \threebars \hbz^{\xi_\d} \threebars}^{2p}
		\aac_p 1 + \threebars \hbz^{\xi_\d} ; \hbz \threebars^{2p} + \threebars \hbz \threebars^{2p}. 
	\end{equation*}
	By~\cite[Thm.'s~$10.7$ and~$10.19$]{hairer_rs}, we then have 
	\begin{equation*}
		\E\sbr[1]{\thinspace \norm[1]{\tilde{u}_{\sh;\d}^{(2)}}_{\CX_T}^p}
		\aac
		1 + \E\sbr[1]{\threebars \hbz^{\xi_\d};\hbz \threebars^{2p}} + \E\sbr[1]{\threebars \hbz \threebars^{2p}}
		\aac 
		1 +\d^{\vartheta p} + 1
		\leq 3
	\end{equation*}
	for each~$0 < \vartheta \ll 1$, so we have established~\eqref{pf:prop:conv_eta:unif_Lp_bdd} and thus that~$\tilde{u}^{(2)}_{\sh;\d} \to  \tilde{u}^{(2)}_{\sh}$ in~$L^p(\Omega;\CX_T)$.
	
	Finally, since we have assumed~$F$ to be Fréchet differentiable in (a neighbourhood of) \mbox{$w_\sh = \Phi(\LL(\sh))$}, its derivative~$DF\sVert[0]_{w_\sh}$ is a bounded, linear operator by definition. As a consequence, we have
	\begin{equation*}
		\abs[0]{\lambda_\d - \lambda}
		\leq
		\E\sbr[1]{\; \abs[1]{DF\sVert[0]_{w_\sh}\del[1]{\tilde{u}^{(2)}_{\sh;\d} - \tilde{u}^{(2)}_{\sh}}}}
		\aac
		\E\sbr[1]{\ \norm[1]{\tilde{u}_{\sh;\d}^{(2)} - \tilde{u}_{\sh}^{(2)}}_{\CX_T}}
		\to 0 \quad \text{as} \quad \d \to 0.
	\end{equation*}
\end{proof}

\begin{remark} \label{rmk:formal_der}
	Considering~$\hat{u}_{\sh;\d}^\eps$ as in~\eqref{eq:gpam_shifted_eps} (with~$\xi \rightsquigarrow \xi_\d$ and~\enquote{$+\infty$} $\rightsquigarrow \fc_\d$) and formally applying~$\partial^i_\eps\sVert[0]_{\eps = 0}$
	leads to the equations for~$\hat{u}_{\sh;\d}^{(i)}$, $i=1,2$, in~\eqref{eq:u_hd1} and~\eqref{eq:u_hd2} above. Note the following:
	\begin{itemize}
		\item For~$i=1$, no products need to be renormalised, so~$\hat{u}_{\sh;\d}^{(1)} = u_{\sh;\d}^{(1)}$, cf.~\cite[Lem.~$2.34$]{friz_klose_22}.
		\item For~$i=2$, the part of~$\hat{u}^{(2)}_{\sh;\d}$ which contains~$\sbr[0]{u_{\sh;\d}^{(1)}}^2$ (denoted by $v_{\sh,\xi_\d,\xi_\d}^{(\<cms>)}$ above) is already accounted for in~$A_\sh - \tilde{A}_\sh$ since it need to be renormalised.
		This is what distinguishes~$\tilde{u}^{(2)}_{\sh;\d}$ from~$\hat{u}^{(2)}_{\sh;\d}$.
	\end{itemize}
\end{remark}

\appendix

\section{Technical constructions} \label{app:rs_extended}

In this section, we amend the regularity structures setting recalled in~\cite[App.~A]{friz_klose_22}:
We introduce \emph{two} new Cameron-Martin symbols~$H_1 \equiv \<cm>$ and~$H_2 \equiv \<cm2>$ representing some~$h,k \in \CH$.
We restrict to this two-fold extension because it is sufficient for our purposes but the construction we present works equally well for~$n > 2$ new Cameron--Martin symbols.

Let~$\TT = (\CT,\CA,\CG)$ be the usual gPAM regularity structure~(see~\cite[Sec.~$3.1$]{cfg}) and recall the notation from Table~\ref{table:symbols} on page~\pageref{table:symbols} above.

\paragraph*{Vector extension.}
\begin{itemize}[itemsep=5pt]
	\item We extend~$\TT$ to~$\TT[\<cm>,\<cm2>] = (\CT[\<cm>,\<cm2>],\CA[\<cm>,\<cm2>],\CG[\<cm>,\<cm2>])$ in the way described by Cannizzaro, Friz, and Gassiat~\cite[Sec.~$3$]{cfg}. In particular, we set~$H_0 := \<wn>$ for ease of notation and then
	\begin{align*}
		\CT[\<cm>,\<cm2>] & = \scal{\CF[\<cm>,\<cm2>]}, \\
		\CF[\<cm>,\<cm2>] & = \{H_i, \CI(H_i)H_j, X_m H_j, \1, \CI (H_i), X_m: i,j \in \{0,1,2\}, \ m \in \{1,2\}\}
	\end{align*}
	and use the standard graphical notation~$\<1p1g> := \CI(H_2)H_1 \equiv \CI(\<cm2>)\<cm>$ etc.
	The degree is associated in the usual way by
	\begin{equation*}
		\deg(H_i) := -1-\kappa, \quad \deg(\CI(\tau)) := \deg(\tau) + 2, \quad
		\deg(\tau \bar{\tau}) := \deg(\tau) \deg(\bar{\tau})
	\end{equation*}
	for~$i \in \{0,1,2\}$. Furthermore, we define the~\emph{sector}\footnote{See~\cite[Def.~$2.5$]{hairer_rs} for the definition of a~\emph{sector}.}
	\begin{equation*}
		\CU[\<cm>,\<cm2>] 
		:= \{\1, \CI(H_i), X_m: i \in \{0,1,2\}, \ m \in \{1,2\}\}
		= \{\1, \<1>, \<1g>, \<1p>, X_m: \ m \in \{1,2\}\}.
	\end{equation*}
	\item The corresponding space of extended models is denoted by~$\MM[\<cm>,\<cm2>]$. \label{app:extended_models}
\end{itemize}
Recall that~$\CF$ is defined like~$\CF[\<cm>,\<cm2>]$ but where~$i = j = 0$ is the only allowed choice for these parameters. 
\begin{itemize}[itemsep=5pt] \label{def:extension_operator}
	\item We introduce the \emph{extension operator} $E_{(h,k)}: \MM \to \MM[\<cm>,\<cm2>]$ analogously to~\cite[Prop.~$3.10$]{cfg}. In particular, de\-no\-ting~$E_{(h,k)} \bz = (\Pi^{\e_{(h,k)}},f^{\e_{(h,k)}})$ for~$\bz = (\Pi,f) \in \MM$, we set for any~$z \in \R^3$:
	\begin{enumerate}[label=(\arabic*), itemsep=5pt]
		\item $\Pi^{\e_{(h,k)}}_z\sVert[0]_{\CT} \equiv \Pi_z$ and~$f_z^{\e_{(h,k)}}\sVert[0]_{\CT^+} \equiv f_z$,
		\item $\Pi^{\e_{(h,k)}}_z \<cm> := h$ and $\Pi^{\e_{(h,k)}}_z \<cm2> := k$,
		\item for~$\tau \in \CF[\<cm>,\<cm2>] \setminus \CF$ with~$\deg(\tau) < 0$ s.t.~$\tau = \tau_1 \tau_2$ with~$\tau_1 \in \CU[\<cm>,\<cm2>]$ and~$\tau_2 \in \{H_i: i \in \{0,1,2\}\}$,
		\begin{equation*}
			\Pi^{\e_{(h,k)}}_z \tau :=  \del[1]{\Pi^{\e_{(h,k)}}_z \tau_1} \del[1]{\Pi^{\e_{(h,k)}}_z \tau_2}
		\end{equation*}
		\item~$\Pi^{\e_{(h,k)}}$ satisfies the \emph{admissibility criterion}, see~\cite[eq.'s~($3.6$a) and~($3.6$b)]{cfg} and~$f^{\e_{(h,k)}}$ is defined from~$\Pi^{\e_{(h,k)}}$ according to~\cite[eq.'s~($3.7$) --~($3.8$)]{cfg}.
	\end{enumerate}
\end{itemize}

\begin{remark}
	Cannizzaro, Friz, and Gassiat only extend the regularity structure by one CM symbol. However, as we have just seen, it is obvious how to adapt their construction to any finite number of new CM symbols. 
	A proof that this is also possible in the \emph{generic} regularity structures framework -- and not only that for~gPAM -- can be found in the work of Sch\"onbauer~\cite[Sec.~$3$]{schoenbauer}.
\end{remark}

\paragraph*{Extension and translation operators.}
Restricting the above construction to~$i,j \in \{0,1\}$, we obtain the extended regularity structure~$\TT[\<cm>]$ introduced in~\cite[Sec.~$3.2$]{cfg} and the extension operator~$E_h: \MM\to \MM[\<cm>]$ constructed in~\cite[Prop.~$3.10$]{cfg}. We write $E_h\bz = (\Pi^{\e_h},f^{\e_h})$ for~$\bz = (\Pi,f)$.

\begin{itemize}[itemsep=3pt]
	\item We define the abstract linear translation operations 
	\begin{equation*}
		\ft_{\<cm2s>}: \CT \to \CT[\<cm2>], \quad 
		\tilde{\ft}_{\<cm2s>}: \CT[\<cm>] \to \CT[\<cm>, \<cm2>]
	\end{equation*}
	for~$\t \in \{\ft_{\<cm2s>}, \tilde{\ft}_{\<cm2s>}\}$ by
	\begin{equs}
		\tilde{\ft}_{\<cm2s>} \<cm> := \<cm>, \quad
		\t \<wn> & := \<wn> + \<cm2>, \quad 
		\t X^k := X^k, \quad
		\t (\tau \sigma) := (\t \tau) (\t \sigma), \quad
		\t (\CI \tau) := \CI(\t \tau)
	\end{equs}
	where~$\sigma, \tau \in \CT$ if~$\t = \ft_{\<cm2s>}$ and~$\sigma, \tau \in \CT[\<cm>]$ if~$\t = \tilde{\ft}_{\<cm2s>}$.
	\item
	As in~\cite[Prop.~$3.12$]{cfg}, we then consider the translation operator~$T_k: \MM \to \MM$ given by~$\Pi^{\t_k} = \Pi^{\e_k} \circ \ft_k$.
	\item
	Now suppose the new \enquote{base regularity structure} is~$\TT[\<cm>]$, the one-fold extension of~$\TT$.
	The one-fold extension of~$\TT[\<cm>]$, in turn, is~$\TT[\<cm>,\<cm2>]$: We denote the corresponding extension and translation operators by~
	\begin{equation*} \label{def:amended_extension_op}
		\tilde{E}_k: \MM[\<cm>] \to \MM[\<cm>,\<cm2>] \quad \text{and} \quad \tilde{T}_k: \boldsymbol{\CM}[\<cm>] \to \boldsymbol{\CM}[\<cm>]. 
	\end{equation*}
	We write~$\Pi^{\e_h,\tilde{\e}_k}$ for the~$\Pi$-component of~$(\tilde{E}_k \circ E_h)\bz \in \MM[\<cm>,\<cm2>]$ and then have
	\begin{equation*}
		\Pi^{\e_h,\tilde{\t}_k} = \Pi^{\e_h,\tilde{\e}_k} \circ \tilde{\ft}_k
	\end{equation*}	
	for the~$\Pi$-component of~$(\tilde{T}_k \circ E_h) \bz \in \MM[\<cm>]$.
	Note that the operations represented by the superscripts are carried out from left to right.
\end{itemize}

The following lemma verifies the intuitively true statements that $(1)$ the order of translation and extension basically does not matter and $(2)$ that two successive one-fold extensions correspond to a two-fold extension.

\begin{lemma} \label{lem:comm_ext_trans}
	For any~$h,k \in \CH$ we have
	$\tilde{T}_k \circ E_h = E_h \circ T_k$
	as well as~$E_{(h,k)} = \tilde{E}_k \circ E_h$.
\end{lemma}

\begin{proof}
	The proof is a straight-forward application of the definitions. First, note that the operations on both sides of the first claim map from~$\MM$ to~$\MM[\<cm>]$ and from~$\MM$ to~$\MM[\<cm>,\<cm2>]$ in the second assertion.
	
	We distinguish the cases for
	\begin{equation*}
		\tau \in \CF[\<cm>] \equiv \{\<wn>, \<cm>, \<11>, \<1g1>, \<1g1g>, \<11g>, X_m \<wn>, X_m \<cm>, \1, \<1>, \<1g>, X_m: \ m = 1,2\}.
	\end{equation*}
	\begin{itemize}[itemsep=5pt]
		\item Suppose~$\tau$ does not contain an instance of~$\<cm>$. By definition, we then have
		\begin{equation*}
			\Pi_z^{\e_h,\tilde{\t}_k} \tau 
			= 
			\Pi_z^{\e_h,\tilde{\e}_k} \tilde{\ft}_{\<cm2s>} \tau
			= 
			\Pi_z^{\e_k} \ft_{\<cm2s>} \tau
			=
			\Pi_z^{\t_k} \tau
			=
			\Pi_z^{\t_k;\e_h} \tau.
		\end{equation*} 
		\item Suppose~$\tau$ contains at least one instance of~$\<cm>$ but no instance of~$\<wn>$. We then have~$\tilde{\ft}_{\<cm2s>} \tau = \tau$ and thus
		\begin{equation*}
			\Pi_z^{\e_h,\tilde{\t}_k} \tau 
			= \Pi_z^{\e_h,\tilde{\e}_k} \tilde{\ft}_{\<cm2s>} \tau
			= \Pi_z^{\e_h} \tau
			= \Pi_z^{\t_k,\e_h} \tau.
		\end{equation*}
		\item Finally, consider the case when~$\tau$ contains instances of both~$\<wn>$ and~$\<cm>$, i.e.~$\tau \in \{\<11g>, \<1g1>\}$. We only consider~$\tau = \<11g>$\,; the argument for the other symbol is analogous.
		We have
		\begin{equs}
			\Pi_z^{\t_k,\e_h} \<11g>
			& =
			\del[1]{\Pi_z^{\t_k,\e_h} \<1>} \del[1]{\Pi_z^{\t_k,\e_h} \<cm>}
			=
			\del[1]{\Pi_z^{\e_k} \ft_{\<cm2s>} \<1>} \del[1]{\Pi_z^{\t_k,\e_h} \<cm>}
			=
			\del[1]{\Pi_z^{\e_k} \sbr[1]{\<1> + \<1p>}} \del[1]{\Pi_z^{\t_k,\e_h} \<cm>} \\
			& =
			\Pi_z^{\e_h,\tilde{\e}_k} \sbr[1]{\<11g> + \<1p1g>}
			=
			\Pi_z^{\e_h,\tilde{\e}_k} \ft_{\<cm2s>} \<11g>
			=
			\Pi_z^{\e_h,\tilde{\t}_k} \<11g>.
		\end{equs}
	\end{itemize}
	By admissibility, the corresponding~$f$'s are defined from the~$\Pi$'s and the proof of the first claim is complete.
	The second statement can be proved by similar arguments.
\end{proof}

\section{Background on Sobolev and Besov spaces} \label{sec:aux_besov}

In this appendix, we present some results on Sobolev and Besov spaces which we are frequently using throughout the article.	

\paragraph*{Sobolev spaces on~$\T^2$.}
For a given~$\alpha \in \R$, denote by~$H^\alpha(\T^2)$ the space of functions resp. distributions~$f: \T^2 \to \R$ for which the following \emph{norm is finite}:
\begin{equation*}
	\norm[0]{f}_{H^{\alpha}(\T^2)} 
	:=
	\norm[0]{\del[0]{1 + \abs[0]{\,\cdot\,}^2}^{\nicefrac{\alpha}{2}} \hat{f}(\,\cdot\,)}_{\ell^2(\mathbb{Z}^2)}
	=
	\del[4]{\,\sum_{k \in \mathbb{Z}^2} \del[1]{1 + \abs[0]{k}^2}^\alpha \hat{f}(k)}^{\frac{1}{2}} < \infty \,.
\end{equation*}
In particular, we have~$H^0(\T^2) = L^2(\T^2)$.

\paragraph*{Besov spaces on~$\T^2$.}
For~$\alpha \in \R$ and~$p,q \in [1,\infty]$, we define the Besov space~$B_{p,q}^\alpha(\T^2)$ as the \emph{completion of the space of smooth functions}~$\CC^\infty(\T^2)$~\emph{under the norm}
\begin{equation*}
	\norm[0]{f}_{B^\alpha_{p,q}}
	:=
	\norm[1]{\, 2^{j\alpha} \norm[0]{\Delta_j f}_{L^p}\,}_{\ell^q(j \geq -1)}
\end{equation*}
where~$\Delta_j$ represents the $j$-th Littlewood--Paley block, see~\cite[Chap.~2]{bcd} for details.
We also use the shorthand notation~$\CC^\alpha(\T^2) := B^{\alpha}_{\infty,\infty}(\T^2)$.
Note the difference between this definition (\enquote{completion under the norm}) vs. the one for Sobolev spaces (\enquote{finite under the norm}) which ensures that the Besov spaces are Polish.
This is a technical issue that has appeared in many contexts and, in our case, is necessary for the large deviation results in~\cite{hairer_weber_ldp} (which the result in Theorem~\ref{thm:precise_laplace} is based upon) to be applicable.
It has almost no effect on the discussions in this article but for the fact that, for any~$\alpha \in \R$, we have 
\begin{equation*}
	B_{2,2}^\alpha(\T^2) \embed H^\alpha(\T^2) \embed \CC^{\alpha'}(\T^2) 
	\quad \text{for all} \ \alpha' < \alpha - 1\, .
\end{equation*}
In other words:
The second continuous embedding becomes slightly worse compared to the standard Besov embedding results in~\cite[Prop.~2.71]{bcd}, which we recall in the next proposition.

\begin{proposition}[Besov embedding] \label{prop:embed_besov}
	Let~$1 \leq p_1 \leq p_2 \leq \infty$ and~$1 \leq q_1 \leq q_2 \leq \infty$.
	Then, for any~$\alpha \in \R$, we have
	\begin{equation*}
		B_{p_1,q_1}^{\alpha_1}(\T^2) \embed B_{p_2,q_2}^{\alpha - 2\del[0]{\frac{1}{p_1} - \frac{1}{p_2}}}(\T^2) \,.
	\end{equation*}
\end{proposition}

We also frequently refer to the following results about multiplication in Besov spaces and improvement of regularity by the heat semigroup.

\begin{proposition}[Besov multiplication] \label{prop:mult_besov}
	Let~$\alpha, \beta \in \R$ and~$p,p_1,p_2,q_1,q_2 \in [1,\infty]$ be such that
	\begin{equation*}
		\frac{1}{p} = \frac{1}{p_1} + \frac{1}{p_2}, \qquad  \alpha \leq \beta, \qquad \alpha + \beta > 0.
	\end{equation*}
	Then, the mapping~$(f,g) \mapsto f \cdot g$ extends to a continuous bilinear map from~$B^{\alpha}_{p_1,q_1} \x B^{\beta}_{p_2,q_2}$ to~$B^{\alpha}_{p,q_1}$. 
\end{proposition}

\begin{remark} \label{rmk:besov_multiplication}
	Note that the product inherits the \enquote{summability parameter}~$q_1$ of the factor~$f$ that is of lower regularity~$\alpha$, and \emph{not} the summability~$q \in [1,\infty]$ given by~$q^{-1} = q_1^{-1} + q_2^{-1}$, even though the latter is true for the \enquote{resonant} part of the product (see~\cite[Thm.~$2.85$]{bcd}).
	We refer to Broux and Lee~\cite[Sec.~$3.3$]{broux_lee} for a detailed discussion of this issue.
	The same article also is an excellent source for a proof of Proposition~\ref{prop:mult_besov} in case~$\alpha < 0 < \beta$ and, in addition, contains further references to the literature.
	For the more general case~$\alpha \leq \beta$, we direct the reader to the lecture notes of van Zuijlen~\cite[Thm.~$19.7$]{van_zuijlen}.
\end{remark}

The next proposition is copied from~\cite[Prop.~A.$13$]{mourrat_weber_infinity} which also contains references for the proofs.
Recall that~$P_t = e^{t\Delta}$ denotes the heat semigroup.

\begin{proposition}[Heat flow] \label{prop:heat_besov}
	Let~$\alpha,\beta \in \R$ such that~$\alpha \geq \beta$ and~$p,q \in [1,\infty]$. Then, there exists some constant~$C > 0$ such that, uniformly over~$t > 0$, we have the estimate
	\begin{equation*}
		\norm[0]{P_t f}_{B_{p,q}^\alpha} \leq C t^{\frac{1}{2}(\beta - \alpha)} \norm[0]{f}_{B_{p,q}^\beta}.
	\end{equation*}
\end{proposition}

\bibstyle{alpha}
\bibliography{bibliography}

\vspace{3em}
\noindent
\begin{minipage}{.5\textwidth}
	\small
	\textbf{Tom Klose} \\
	University of Oxford \\ 
	Mathematical Institute \\
	Woodstock Road \\
	Oxford, OX2 6GG, United Kingdom \\
	{\it E-mail address:}
	{\tt tom.klose@maths.ox.ac.uk}
\end{minipage}%

\end{document}